\documentclass{jmd}
\usepackage{amscd}
\usepackage[colorlinks]{hyperref}
\hypersetup{linkcolor=blue, citecolor=red}

\issueinfo[P]{5}{2}{285} 

\subjclass{Primary: 37D25, 37D40; Secondary: 14F05}
\keywords{Teichm\"uller geodesic flow, Kontsevich--Zorich cocycle, square-tiled surfaces}

\date{July, 24, 2010}
\revised{June 1, 2011}
\commby{Rapha\"el Krikorian}

\title[Square-tiled cyclic covers]
      {Square-tiled cyclic covers}

\author[Giovanni Forni]{Giovanni Forni}
\address{Department of Mathematics, University of Maryland, College Park, MD 20742-4015, USA}
\email{gforni@math.umd.edu}
\thanks{GF: Partially supported by the National Science Foundation under grant
DMS 0800673.}

\author[Carlos Matheus]{Carlos Matheus}
\address{Coll\`ege de France, 3 Rue d'Ulm, Paris, CEDEX 05, France}
\curraddr{CNRS, LAGA, Institut Galil\'ee, Universit\'e Paris 13, 99, Av. Jean-Baptiste Cl\'ement, 93430,
Villetaneuse, France}
\email{matheus@impa.br}
\urladdr{http://w3.impa.br/ \~ cmateus}
\thanks{CM: Partially supported by Coll\`ege de France and French ANR (grant 0864 Dynamique dans l'espace de Teichm\"uller).}

\author[Anton Zorich]{Anton Zorich}
\address{IRMAR, Universit\'e de Rennes 1, Campus de Beaulieu, 35042, Rennes, France}
\email{Anton.Zorich@univ-rennes1.fr.}
\urladdr{http://perso.univ-rennes1.fr/anton.zorich/}

\newtheorem{theorem}{Theorem}
\newtheorem{corollary}[theorem]{Corollary}

\newtheorem{lemma}[theorem]{Lemma}
\newtheorem{proposition}[theorem]{Proposition}
\newtheorem{conjecture}[theorem]{Conjecture}
\newtheorem*{problem}{Problem}
\theoremstyle{definition}
\newtheorem{definition}[theorem]{Definition}
\newtheorem{remark}[theorem]{Remark}
\newtheorem{example}[theorem]{Example}

\newcommand{\ep}{\varepsilon}
\newcommand{\eps}[1]{{#1}_{\varepsilon}}

\begin{document}

\begin{abstract}
A  cyclic  cover  of  the  complex  projective  line branched at four
appropriate points has a natural structure of a square-tiled surface.
We  describe  the  combinatorics  of such a square-tiled surface, the
geometry  of  the  corresponding Teichm\"uller curve, and compute the
Lyapunov  exponents  of the determinant bundle over the Teichm\"uller
curve  with  respect  to the geodesic flow. This paper includes a new
example  (announced  by G. Forni and C. Matheus in \cite{Forni:Matheus})
of  a Teichm\"uller curve of a square-tiled cyclic cover in a stratum
of  Abelian  differentials  in genus four with a maximally degenerate
Kontsevich--Zorich  spectrum (the only known example found previously
by Forni in genus three also corresponds to a square-tiled cyclic cover \cite{ForniSurvey}).
 We  present  several new examples of Teichm\"uller curves in
strata  of  holomorphic  and meromorphic quadratic differentials with
maximally  degenerate  Kontsevich--Zorich spectrum. Presumably, these
examples  cover  all  possible  Teichm\"uller  curves  with maximally
degenerate spectrum. We prove that this is indeed the case within the
class of square-tiled cyclic covers.
\end{abstract}


\maketitle

\section{Introduction}
The  Kontsevich--Zorich  cocycle  is  a dynamical system on the total
space  of  the  Hodge  bundle  over  the  moduli  space of Abelian or
quadratic  differentials.  It  is  a  continuous-time  cocycle in the
standard  sense: it is a flow which acts linearly on the fibers
of  the  bundle.  Its  projection to the moduli space is given by the
Teichm\"uller  geodesic flow. Since the Kontsevich--Zorich cocycle is
closely related to the tangent cocycle of the Teichm\"uller flow, its
Lyapunov  structure determines that of the Teichm\"uller flow and has
implications  for  its dynamics. The Kontsevich--Zorich spectrum also
plays  a  crucial role in applications of \mbox{Teichm\"uller} theory to the
dynamics  of translation flows and interval exchange transformations ---
in  particular,  to  results  on  the  deviation  of  ergodic averages
(see \cite{Zorich1},  \cite{Zorich2},  \cite{Zorich3},  \cite{Kontsevich},
\cite{Forni}),   on  existence  and  nature  of  the limit distributions
(see \cite{Bufetov1},   \cite{Bufetov2}) and on the weak mixing
property (see \cite{Avila:Forni}).

M. Kontsevich and A. Zorich conjectured that the Lyapunov spectrum of
the  cocycle is simple (in particular, that all the exponents are non-zero)
for $\textrm{SL}(2,\mathbb{R})$-invariant canonical absolutely continuous measures on
all connected components of strata  of  the  moduli spaces of Abelian
(and quadratic) holomorphic differentials.  G. Forni \cite{Forni}  proved that,
for the canonical measures on strata of \emph{Abelian differentials}, the
exponents are all non-zero.   A. Avila   and   M. Viana \cite{Avila:Viana}  later
completed  the  proof  of  the  Kontsevich--Zorich conjecture in this
case.  Recently, G. Forni developed his approach  to  give  a
general criterion for the non-vanishing of the Kontsevich--Zorich  exponents
for  $\textrm{SL}(2,\mathbb{R})$-invariant  measures  on the moduli  space  of  Abelian  differentials \cite{ForniCriterion}.
Based on this criterion (and on a standard construction of an orienting double cover), R. Trevi\~ no
\cite{Trevino} proved the nonvanishing of the exponents for all canonical  measures
on  strata  of quadratic differentials. The full Kontsevich--Zorich   conjecture   is  still  open
for  strata  of nonorientable quadratic differentials.

The  aforementioned results lead to the natural questions as to whether or not it
is possible for the Kontsevich--Zorich cocycle to have zero exponents
with respect to other invariant measures. It is well-known to experts (\textit{cf}. \cite{Veech82})
that  it  is  possible  to construct invariant measures \emph{for the
Teichm\"uller  geodesic  flow  }(for  instance, supported on periodic
orbits)  with  maximally degenerate Kontsevich--Zorich spectra, that
is,  with  all  exponents  equal  to  zero  with the exception of the
``trivial'' ones. Answering a question of W. Veech, G. Forni found in
\cite{ForniSurvey}  the  first  example  of a $\textrm{SL}(2,\mathbb{R})$-invariant measure
with a maximally  degenerate  spectrum.  The  example  is given by the
measure  supported  on  the $\textrm{SL}(2,\mathbb{R})$-orbit of a genus three square-tiled
cyclic cover, that is, a branched cover of the four-punctured Riemann
sphere,  endowed with a quadratic differential with four simple poles
at  the  punctures. Any cover of this type is ``parallelogram-tiled''
in  the  sense  that  it is also a branched cover of the torus with a
single  branching  point.  The flat surface in Forni's example, found
independently  by  F. Herrlich,  M. M\"oller  and G. Schmithuesen, is
very  peculiar, very symmetric, and has so many remarkable properties
that       has      been      aptly      named      \emph{Eierlegende
Wollmilchsau} \cite{Herrlich:Schmithuesen}.    Later  G. Forni   and
C. Matheus  announced  in  the preprint \cite{Forni:Matheus} a second
example    of    the    same   kind   in   genus   four   (see   also
\cite{Matheus:Yoccoz}).  The  present article has, in fact, grown out
of that announcement.

M. M\"oller   conjectured in \cite{Moeller} that  these  two  examples  are  the  only
Teichm\"uller  curves  with  maximally  degenerate Kontsevich--Zorich
spectra.  He  was able to prove his conjecture up to a few strata in
genus  five  where  the  arithmetic conditions he derives to rule out
maximally  degenerate  spectra could not be verified \cite{Moeller}.
M\"oller  result  naturally  led  to  the  more general conjecture on
whether  the Teichm\"uller curves of the Eierlegende Wollmilchsau and
the   Forni--Matheus   curve   indeed   give   the  only  examples  of
$\textrm{SL}(2,\mathbb{R})$-invariant  measures  (or  even  of $\textrm{SL}(2,\mathbb{R})$-orbits) with maximally
degenerate   Kontsevich--Zorich   spectra   on   strata  of  Abelian
differentials.  For  sufficiently high genus the conjecture is proved
in \cite{Eskin:Kontsevich:Zorich} for $\textrm{SL}(2,\mathbb{R})$-invariant suborbifolds in
the  moduli space of holomorphic Abelian and quadratic differentials,
as  a  corollary of the key formula for the sum of the
exponents.  Other  results  in  this  direction  for moduli spaces of
holomorphic  (Abelian  or quadratic) differentials in all genera have
been  announced  by  A. Avila  and  M. M\"oller and, independently,
by  D. Aulicino. Similar conjectures  for strata of \emph{meromorphic }quadratic
differentials are at the moment wide open, to the authors' best knowledge.

In  this  paper, we systematically investigate these questions within
the  class of all square-tiled cyclic covers. We remark that the idea
of  the  construction of a square-tiled cyclic cover already appeared
in  \cite{ForniSurvey}  (and  in  \cite{Forni:Matheus}) but only in a
particular  case.  Here we generalize the construction and derive the
main  topological,  geometric,  and  combinatorial  properties  of the
resulting  translation  and  half-translation  surfaces.  We then
classify  all  the  examples  of  Teichm\"uller  curves  derived  from
square-tiled  cyclic  covers  with  maximally  degenerate spectra in
strata   of   Abelian  holomorphic  differentials  and  of  quadratic
holomorphic  and  meromorphic  differentials.  The  main  tool in our
investigation    of    the spectrum of Lyapunov exponents is   the
formula from \cite{Eskin:Kontsevich:Zorich} for the sum
of   the   non-negative  Lyapunov  exponents (that is, for the exponent
of the determinant bundle).  The  formula  takes  a particularly simple,
\emph{explicit} form in the case of square-tiled cyclic  covers.  In fact, our paper
can be considered as a companion to    the    paper   by   \mbox{A. Eskin},   M. Kontsevich
and   A. Zorich \cite{Eskin:Kontsevich:Zorich:cyclic}  in  which  the  authors derive
a completely explicit formula for each individual Lyapunov exponent
of a square-tiled cyclic cover. In a related paper D. Chen \cite{Chen}
computes   the   Lyapunov  exponent  of  the  determinant  bundle for
square-tiled cyclic covers by different methods and relates it to the slope
of the corresponding Teichm\"uller curve.

Finally,   we  remark  that  the  Eierlegende  Wollmilchsau  and  the
Forni--Matheus  example were originally found in
\cite{ForniSurvey}  and  \cite{Forni:Matheus},  respectively, by a completely
different method  based  on  the  analysis of the action of the cyclic group of
deck  transformations  on  the  second  fundamental form of the Hodge
bundle   (related  to  the  Kontsevich--Zorich  spectrum  by  the
variational  formulas of \cite{Forni}, \cite{ForniSurvey}). This symmetry
approach  has  led  us  to  conduct a  systematic investigation of the
spectrum  of the Kontsevich--Zorich cocycle on equivariant subbundles
of   the   Hodge   bundle,   which will appear  in  a  forthcoming
paper \cite{CMFZ-survey}.

\smallskip
\noindent
\textbf{Additional bibliographic remarks.}
Cyclic covers  were already  studied by I. Bouw and M. M\"oller  in
\cite{Bouw:Moeller}  in  a similar context,  but with respect to
completely  different  (not square-tiled) flat  structures. The
papers \cite{Bouw}  of  I. Bouw and \cite{McMullen}  of  \mbox{C. McMullen}
investigate more general cyclic covers, but without any relation to
flat metrics.  The paper \cite{Wright} of A. Wright   studies general
square-tiled  \textit{Abelian}  (versus \textit{cyclic}) covers.

\subsection{Reader's guide}

Cyclic  covers  are  defined  in Section \ref{ss:Cyclic:covers}. In Section
 \ref{ss:Square:tiled:four:Cyclic:cover} we introduce a square-tiled flat structure
on any appropriate cyclic cover. A reader interested in the main
results can then choose to pass directly to Section \ref{ss:Sum:of:exponents}.
In Section \ref{ss:Singularity:pattern} we characterize square-tiled
cyclic  covers defined by holomorphic one-forms. We
determine  the  corresponding  ambient  strata of holomorphic $1$-forms
(of quadratic differentials in the general situation). In Section \ref{ss:Combinatorics:of:square-tiled} we describe in detail
how  to  explicitly  construct the square-tiled surface $(M_N(a_1,a_2,a_3,a_4),p^\ast
q_0)$, and in Section \ref{ss:Teichmuller:curve:associated:to:a:cyclic:cover}
we characterize   its  Veech  group  and  the  corresponding  arithmetic
Teichm\"uller curve. In Section \ref{sec:symmetries} we describe the automorphism group
of a cyclic cover.  In Section \ref{ss:sum:general}  we recall a general formula from \cite{Eskin:Kontsevich:Zorich}  for the sum of positive Lyapunov exponents of the Hodge bundle
over an arithmetic Teichm\"uller curve. From this formula we derive in
Section \ref{ss:Sum:of:exponents} an explicit  expression  for  the  sum  of  exponents
in the case of an arbitrary  square-tiled  cyclic  cover.  We  apply  these  results  in
Section \ref{ss:degenerate:spectrum} to determine square-tiled cyclic
covers  giving  rise  to arithmetic Teichm\"uller curves with a maximally
degenerate Kontsevich--Zorich spectrum.

In Appendix \ref{a:parity:of:spin} we present an analytic computation
of  the  spin-structure  (different from the  original  computation
in \cite{Matheus:Yoccoz}),  which  allows  us to determine the connected
component  of  the  ambient  stratum corresponding to the exceptional
square-tiled cyclic cover in genus four. Appendix \ref{a:exercise} is provided for the
sake of completeness: it is a technical exercise related to the proof
of one of the main Theorems (namely Theorem \ref{thm:1prime}) .
\smallskip

\section{Square-tiled flat structure on a cyclic cover}
\label{s:Square:tiled:structure}
\subsection{Cyclic covers}
\label{ss:Cyclic:covers}

Consider  an  integer $N$ such that $N>1$ and a $4$-tuple of integers
$(a_1,\dots,a_4)$ satisfying the following conditions:
\begin{equation}
\label{eq:a1:a4}
0<a_i\le N\,;\quad
\gcd(N, a_1,\dots,a_4) =1\,;\quad
\sum\limits_{i=1}^4 a_i\equiv 0 \ (\textrm{ mod } N)\ .
\end{equation}
Let    $z_1,z_2,z_3,z_4\in    \mathbb{C}$    be    four    distinct   points.
Conditions \eqref{eq:a1:a4}    imply    that,    possibly   after   a
desingularization, a Riemann surface  $M_N(a_1,a_2,a_3,a_4)$
defined by equation
\begin{equation}
\label{eq:wN:z14}
w^N=(z-z_1)^{a_1}(z-z_2)^{a_2}(z-z_3)^{a_3}(z-z_4)^{a_4}
\end{equation}
is  closed, connected and nonsingular. By construction,
$M_N(a_1,a_2,a_3,a_4)$ is a  ramified cover over the Riemann sphere
$\mathbb{P}^1(\mathbb{C})$ branched over the  points  $z_1, \dots$, $
z_4$. By puncturing the ramification points we obtain a regular
$N$-fold cover over
$\mathbb{P}^1(\mathbb{C})\setminus\{z_1,z_2,z_3,z_4\}$.

\begin{remark}
It  is  easy  to see that quadruples $(a_1,a_2,a_3,a_4)$ and $(\tilde
a_1,\tilde a_2,\tilde a_3,\tilde a_4)$ with $a_i=\tilde a_i\pmod N$
for  $i=1,2,3,4$, define isomorphic cyclic covers, which explains the
first   condition   in formula \eqref{eq:a1:a4}.   The  condition  on  $\gcd$
in \eqref{eq:a1:a4}  is  a  necessary  and  sufficient  condition  of
connectedness  of  the  resulting  cyclic  cover. The third condition
in formula \eqref{eq:a1:a4} implies that there is no branching at infinity.
\end{remark}

A  group  of  deck  transformations of this cover is the cyclic group
$\mathbb{Z}/N\mathbb{Z}$ with a generator $T:M\to M$ given by
\begin{equation}
\label{eq:T}
T(z,w)=(z,\zeta w) \,,
\end{equation}
where  $\zeta$  is  a  primitive  $N$th  root  of unity, $\zeta^N=1$.
Throughout  this  paper we will use the term \textit{cyclic cover} when referring to a Riemann
surface $M_N(a_1,\dots,a_4)$, with parameters $N,a_1,\dots,a_4$
satisfying relations \eqref{eq:a1:a4}.

\subsection{Square-tiled surface associated to a cyclic cover}
\label{ss:Square:tiled:four:Cyclic:cover}

Any  meromorphic  quadratic  differential  $q(z)(dz)^2$  with at most
simple poles on a Riemann surface defines a flat metric $g(z)=|q(z)|$
with  conical  singularities  at  zeroes and poles of $q$. Let us consider
a meromorphic quadratic differential $q_0$ on $\mathbb{P}^1(\mathbb{C})$ of the form
\begin{equation}
\label{eq:q:on:CP1}
q_0:=\frac{
c_0
(dz)^2}{(z-z_1)(z-z_2)(z-z_3)(z-z_4)}\,,
\quad\text{where } c_0\in\mathbb{C}\setminus\{0\}\,.
\end{equation}
 It  has  simple poles at $z_1,z_2,z_3,z_4$ and no
other  zeroes  or  poles.  The quadratic differential $q_0$ defines a
flat  metric  on  a  sphere  obtained  by the following construction.
Consider  an appropriate flat cylinder. On each boundary component of
the  cylinder  mark  a pair of opposite points and glue the resulting
pairs  of  cords by isometries. The four marked points become conical
points  of  the  flat  metric.  For a convenient choice of parameters
$c_0$,
$z_1,\dots,z_4$   the  resulting  flat  sphere  can  be  obtained  by
identifying   the   boundaries  of  two  copies  of  a  unit  square.
Metrically, we get a square pillow with four corners corresponding to
the four poles of $q_0$, see Figure \ref{fig:pillow}.

Now  consider  some  cyclic cover $M_N(a_1,a_2,a_3,a_4)$ and the canonical projection
$p:M_N(a_1,a_2,a_3,a_4)\to\mathbb{P}^1(\mathbb{C})$.  Consider  an  induced quadratic differential
$q=p^\ast  q_0$  on  $M_N(a_1,a_2,a_3,a_4)$  and  the  corresponding  flat metric. By
construction,  the resulting flat surface is tiled with unit squares.
In   other   words,   we   get   a   \textit{square-tiled   surface},
see \cite{Eskin:Okounkov}, \cite{Zorich:square:tiled} (also called an
\textit{origami},  \cite{Lochak},  \cite{Schmithusen}; also called an
\textit{arithmetic translation surface}, see \cite{Gutkin:Judge}).

  In    this    paper    we    mostly focus on
\textit{square-tiled  cyclic  covers},  which are pairs
$$
(M_N(a_1,a_2,a_3,a_4)\ ,\ p^\ast q_0)\,,
$$
where the meromorphic quadratic differential $q_0$ on the underlying $\mathbb{P}^1(\mathbb{C})$
defines  a ``unit square pillow with vertical and horizontal sides''
as in Figure \ref{fig:pillow}.

\subsection{Singularity pattern of a square-tiled cyclic cover}
\label{ss:Singularity:pattern}

The Riemann surface
$M_N(a_1,a_2,a_3,a_4)$  has  $\gcd(N,a_i)$  ramification  points over
each branching point $z_i\in \mathbb{P}^1(\mathbb{C})$,
where $i=1,2,3,4$, on the base sphere.
Each  ramification point has degree $N/\gcd(N,a_i)$. The flat
metric  has  four  conical  points $z=z_i$, $i=1,\dots,4$, on the base
sphere  with  a  cone  angle  $\pi$ at each conical point. Hence, the
induced  flat  metric  on  $M_N(a_1,a_2,a_3,a_4)$  has  $\gcd(N,a_i)$
conical  points  over  $z_i$;  each  conical  point  has  cone  angle
$\big(N/\gcd(N,a_i)\big)\pi$.

If  one of the cone angles is an odd multiple of $\pi$, then the flat
metric  has  nontrivial  holonomy; in other words, the quadratic
differential  $q=p^\ast  q_0$  \textit{is  not}  a global square of a
holomorphic  $1$-form.  Note, however, that although the condition that all cone
angles be even  multiples  of  $\pi$  is  necessary, it is not
a sufficient condition for triviality of the holonomy of the flat metric.

Denote  by  $\mathcal{H}(m_1,\dots,m_n)$ the stratum of Abelian differentials
with  zeroes  of  degrees $m_1,\dots,m_n$ and by $\mathcal{Q}(d_1,\dots,d_n)$
the stratum of meromorphic quadratic differentials with singularities
of  degrees  $d_1,\dots,d_n$. Here $m_i\in\mathbb{N}$, for $i=1,\dots,n$, and
$\sum_{i=1}^n  m_i  =  2g-2$.  We do not allow poles of orders higher
than    one    for    meromorphic    quadratic    differentials,   so
$d_i\in\{-1,1,2,3,\dots\}$,  for $i=1,\dots,n$, and the
sum of degrees of singularities satisfies the equality
$\sum_{i=1}^n d_i =4g-4$.

\begin{lemma}
\label{lm:singularity:pattern}
If  $N$  is  even  and all $a_i$, $i=1,2,3,4$, are odd, the quadratic
differential  $q=p^\ast  q_0$  is  a  global  square of a holomorphic
$1$-form $\omega$ on $M_N(a_1,a_2,a_3,a_4)$, where $\omega$ belongs to the stratum
\begin{equation}
\label{eq:singularities:one:form}
\omega\in\mathcal{H}\Bigg(
\underbrace{\frac{N}{2\gcd(N,a_1)}-1,\dots\,}_{\gcd(N,a_1)},
\dots,
\underbrace{\,\dots,\frac{N}{2\gcd(N,a_4)}-1}_{\gcd(N,a_4)}\Bigg)\ .
\end{equation}
The  associated flat metric on such a square-tiled cyclic cover has a
trivial linear holonomy.

If  $N$  is  odd,  or  if  $N$  is  even  but  at least one of $a_i$,
$i=1,2,3,4$,  is also even, the quadratic differential $q=p^\ast q_0$
is not a global square of a holomorphic $1$-form on $M_N(a_1,a_2,a_3,a_4)$. In
this case $q$ belongs to the stratum
\begin{equation}
\label{eq:singularities:quadratic:differential}
q\in\mathcal{Q}\Bigg(
\underbrace{\frac{N}{\gcd(N,a_1)}-2,\dots\,}_{\gcd(N,a_1)},
\dots,
\underbrace{\,\dots,\frac{N}{\gcd(N,a_4)}-2}_{\gcd(N,a_4)}
\Bigg)\ ,
\end{equation}
and  the  flat  metric on such a square-tiled cyclic cover $M_N(a_1,a_2,a_3,a_4)$ has
nontrivial linear holonomy.
\end{lemma}

\begin{remark}
In  the  case,  when  $\gcd(N,a_i)=\frac{N}{2}$,  the corresponding ``conical
points''  have  cone  angles  $2\pi$ and so are actually regular
points  of  the  metric.  Depending on the context we either consider
such points as marked points or simply ignore them.
\end{remark}

\begin{remark}
\label{rm:holomorphic:meromorphic}
Lemma \ref{lm:singularity:pattern}  implies,  in particular, that the
quadratic  differential $q=p^\ast q_0$ is \textit{holomorphic} if and
only  if  inequalities  $a_i\le N$ are strict for all $i=1,2,3,4$. If $a_i=N$
for   at   least  one  index  $i$, then the  quadratic
differential  $q=p^\ast  q_0$ is \textit{meromorphic}; that is, it has
simple poles.
\end{remark}

\begin{proof}[Proof of the Lemma]
Let  $\sigma_i$  be  a  small  contour  around  $z_i$ in the positive
direction  on  the  sphere (see  Figure \ref{fig:pillow}).  The paths
$\sigma_i$,  $i=1,2,3$  generate  the fundamental group of the sphere
punctured at the four ramification points.

Since  the  cone  angle  at  each  cone singularity of the underlying
``flat  sphere'' is $\pi$ (whether it is glued from squares
or  not),  the parallel transport along each loop $\sigma_i$ brings a
tangent vector $\vec v$ to $-\vec v$. Let $z\in\mathbb{P}^1(\mathbb{C})$ be a point of the
loop $\sigma_i$, and let $(w,z)$ be one of its preimages in the cover
$M_N(a_1,a_2,a_3,a_4)$.  By lifting  the  loop  $\sigma_i$ to a path on the cover which
starts at the point $(w,z)$, we land at the point $(\zeta^{a_i} w,z)$,
where $\zeta$ is the primitive $N$th root of unity. Thus
we  get  the following representation of the fundamental group of the
punctured   sphere   in   the   cyclic   group   $\mathbb{Z}/N\mathbb{Z}$   of   deck
transformations \eqref{eq:T}  and  the holonomy group $\mathbb{Z}/2\mathbb{Z}$ of the
flat metric on the sphere:
\begin{equation}
\label{eq:representations}
\operatorname{Deck}: \sigma_i\mapsto a_i \in\mathbb{Z}/N\mathbb{Z}\qquad
\operatorname{Hol}: \sigma_i\mapsto 1\in \mathbb{Z}/2\mathbb{Z}
\end{equation}
Since  the metric on $M_N(a_1,a_2,a_3,a_4)$ is induced from the metric on the sphere,
the  holonomy  representation $\textrm{hol}$ of  the fundamental group of the covering surface
$M_N(a_1,a_2,a_3,a_4)$ factors through the one of the sphere, \textit{i.e.}, $\textrm{hol}=\textrm{Hol}\circ p_{\ast}$.

Let us suppose that $N$ is odd. Let $\alpha_i=N/\gcd(N,a_i)$. Then
$$
\operatorname{Deck}(\sigma_i^{\alpha_i})=a_i\cdot\frac{N}{\gcd(N,a_i)}=
\frac{a_i}{\gcd(N,a_i)}\,N=0\pmod N\ .
$$
Hence $\sigma_i^{\alpha_i}$ can be lifted to a closed path on $M_N(a_1,a_2,a_3,a_4)$.
On                  the                  other                  hand,
$\operatorname{Hol}(\sigma_i^{\alpha_i})=\alpha_i=N/\gcd(N,a_i)=1\pmod   2$.  Thus,
the  flat metric on
the cover
$M_N(a_1,a_2,a_3,a_4)$ has nontrivial holonomy, and our quadratic
differential  $q$  is  \textit{not}  a global square of a holomorphic
$1$-form.

Let us suppose that $N$  is  even,  but  some  $a_i$  is  also even. Since
$$
\gcd(N,a_1,\dots,a_4)=1\,,
$$
at  least  one of $a_j$ is odd. Since $a_1+\dots+a_4$ is divisible by
$N$,  this  sum  is even, and hence we have exactly two even entries.
By relabeling the ramification points on the sphere, if necessary, we
may assume that $a_1, a_3$ are odd while $a_2,a_4$ are even. Let us consider
a path
$$
\rho:=\sigma_1^{(a_2+a_3)}\sigma_2^{-a_1}\sigma_3^{-a_1}
$$
For the induced deck transformation we get
$$
\operatorname{Deck}(\rho)=\left((a_2+a_3)a_1-a_1 a_2-a_1 a_3\right)=0\ .
$$
Hence,  a lift of $\rho$ is closed on $M_N(a_1,a_2,a_3,a_4)$. On the other hand, this
path  acts  as  an element $\operatorname{Hol}(\rho)=(a_2+a_3)-a_1-a_1=1\pmod 2$ in
the  holonomy group $\mathbb{Z}/2\mathbb{Z}$. Hence, the holonomy of the metric along
this closed path on the cover $M_N(a_1,a_2,a_3,a_4)$ is nontrivial, and our quadratic
differential  $q$  is  \textit{not}  a global square of a holomorphic
$1$-form.

Finally, suppose that $N$ is even and all $a_i$, $i=1,2,3,4$ are odd.
Any element of the fundamental group of the underlying four-punctured
sphere can be represented by a product
$$
\tau=\prod_{j=1}^n\sigma_{i_j}^{p_j}\,,
\quad\text{where } p_j=\pm 1\text{ and } i_j\in\{1,2,3,4\}\,.
$$
Let  $k_i$  be  the  algebraic  number  of  entries of the ``letter''
$\sigma_i$  in  the  word  as above, where $\sigma^{-1}_i$ is counted
with the sign minus, and $i=1,2,3,4$. The loop $\tau$ as above can be
lifted  to a closed loop on the cover if and only if $\operatorname{Deck}(\tau)=0$,
that is,
if  and  only  if  $k_1  a_1+\dots  +k_4  a_4=0\pmod N$. In this case
$\operatorname{Hol}(\tau)=k_1+\dots+k_4$.   Since   all   $a_i$  are  odd  we  have
$k_1+\dots+k_4=k_1  a_1+\dots +k_4 a_4\pmod 2$. Finally, since $N$ is
even  and  $k_1  a_1+\dots  +k_4  a_4=0\pmod  N$,  we  also have $k_1
a_1+\dots  +k_4  a_4=0\pmod  2$.  Thus, in this case, the flat metric
on $M_N(a_1,a_2,a_3,a_4)$ has trivial linear holonomy.

Recall  that  a  flat  metric  associated  to a meromorphic quadratic
differential with at most simple poles has cone angle $(d+2)\pi$ at a
zero  of  order  $d$ (we consider a simple pole as a ``zero of degree
$-1$'').  A  flat  metric associated to a holomorphic $1$-form has cone
angle  $2(k+1)\pi$ at a zero of degree $k$. We have already evaluated
the  number  of  ramification  points  and their ramification degrees
which  define cone angles at all conical singularities, and hence the
degrees  of  the corresponding quadratic (Abelian) differential. This
completes the proof of the Lemma.
\end{proof}

\begin{lemma}
\label{lm:minus:omega}
In   the   case   when  the  quadratic  differential  $q=p^\ast  q_0$
determining  a  square-tiled  flat  structure is the global square of a
holomorphic  $1$-form, that is, $q=\omega^2$, the form $\omega$ is anti-invariant
with  respect  to  the  action  of  a  generator of the group of deck
transformations,
\begin{equation}
\label{eq:omega:antiinvariant}
T^\ast\omega=-\omega\,.
\end{equation}
\end{lemma}
\begin{proof}
By   construction,  the  quadratic  differential  $q=p^\ast  q_0$  is
invariant  under the action of deck transformations on $M_N(a_1,a_2,a_3,a_4)$. Hence,
when  $N$  is  even,  all  $a_i$,  $i=1,2,3,4$ are odd, and $q=p^\ast
q_0=\omega^2$, the holomorphic $1$-form $\omega$ is either invariant or
anti-invariant  under  the  action of a generator of the group of deck
transformations \eqref{eq:T}.

Invariance of $\omega$ under $T^\ast$ would mean that $\omega$ can be
pushed forward to $\mathbb{P}^1(\mathbb{C})$, which would imply in turn that $q_0$
is  a  global  square  of  a  holomorphic $1$-form. This is not true.
Hence, $\omega$ is anti-invariant.
\end{proof}

By Riemann--Hurwitz formula, the genus $g$ of
$M_N(a_1,a_2,a_3,a_4)$ satisfies
\begin{equation*}
\begin{aligned}
2-2g&=2N-
\sum_{\substack{\text{ramification}\\
                \text{points}}}
(\text{degree of ramification}-1)
\\&=\
2N-\sum_{i=1}^4 \gcd(N,a_i)\cdot\big(N/\gcd(N,a_i)-1\big) \\
&=\ \sum_{i=1}^4 \gcd(N,a_i)-2N
\end{aligned}
\end{equation*}
and hence,
\begin{equation}
\label{eq:genus}
g=N+1- \frac{1}{2} \sum\limits_{i=1}^4\textrm{gcd}(a_i,N)\,.
\end{equation}

The  same  result can be obtained by summing up the degrees of zeroes,
which gives $2g-2$ for a holomorphic $1$-form in \eqref{eq:singularities:one:form},
and  $4g-4$  for  a  quadratic differential in \eqref{eq:singularities:quadratic:differential}.

\subsection{Combinatorics of a square-tiled  cyclic cover}
\label{ss:Combinatorics:of:square-tiled}

It  is convenient to define a square-tiled surface corresponding to a
holomorphic  $1$-form  by a pair of permutations on the set of all squares, $\pi_h$ and $\pi_v$, indicating for each square its neighbor to the right and its neighbor on top respectively.
Let us evaluate these permutations for a square-tiled surface defined  by  a  holomorphic
$1$-form  $\omega$  on $M_N(a_1,a_2,a_3,a_4)$ (and in addition to the conditions in formula \eqref{eq:a1:a4}  we  assume  that  $N$ is even and all $a_i$ are odd).

We start with an appropriate enumeration of the squares.
By  construction, our ``square-tiled pillow'' $(\mathbb{P}^1(\mathbb{C}),q_0)$ in the base
of the cover
$$
M_N(a_1,a_2,a_3,a_4)\to\mathbb{P}^1(\mathbb{C})
$$
is tiled with two unit squares. Since the above cover has degree $N$,
our square-tiled cyclic cover gets tiled with $2N$ squares.

Assume  that  the  branch points  $z_1,z_2,z_3,z_4$  are  associated  to the
corners  of  the  pillow  as in Figure \ref{fig:pillow}. We associate the
letters  $A,B,C,D$  to  the  four  corners  respectively.  We also
associate  the corresponding letters to the corners of each square on
the surface $M_N(a_1,a_2,a_3,a_4)$.

Let  us  color one of the faces of our pillow in white, and the other
one  in black. Let us lift this coloring to $M_N(a_1,a_2,a_3,a_4)$. Choose some white
square, and associate  the  number  $0$  to it. Take a black square
adjacent  to  the  side  $[CD]$  of  the first one and associate the
number  $1$  to  it. Acting by deck transformations we associate to a
white  square  $T^k(S_0)$  the  number  $2k$,  and  to a black square
$T^k(S_1)$  the  number $2k+1$. As usual, $k$ is taken modulo $N$, so
we may assume that $0\le k< N$.

\begin{figure}[hbt]
\includegraphics{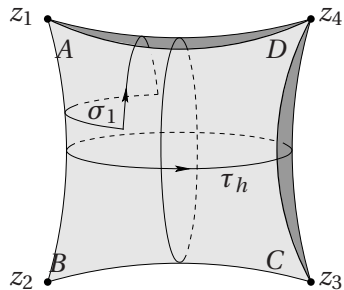}
\begin{picture}(0,0)(100,0)
\put(20,-10){$z_1$}
\put(20,-110){$z_2$}
\put(137,-110){$z_3$}
\put(137,-10){$z_4$}
\put(37,-24){$A$}
\put(35,-105){$B$}
\put(117,-104){$C$}
\put(117,-24){$D$}
\put(50,-47){$\sigma_1$}
\put(100,-73){$\tau_h$}
\end{picture}
\vspace{110bp}
\caption{
\label{fig:pillow}
Flat sphere glued from two squares
}
\end{figure}

Let us consider  a  small  loop  $\sigma_i$  encircling  $z_i$ in a positive
direction; we assume that $\sigma_i$ does not have other ramification
points  inside  an  encircled  domain,  see  Figure \ref{fig:pillow}.
Let us then consider  a  lift of $\sigma_i$ to $M_N(a_1,a_2,a_3,a_4)$. The end-point of the lifted
path is the image of the action of $T^{a_i}$ on the starting point of
the lifted path. Hence, starting at a square number $j$ and ``going
around   a   corner''   on  $M_N(a_1,a_2,a_3,a_4)$  in  the  positive 
(counterclockwise)  direction  we  get to a square number $j+2a_j\pmod
{2N}$ (see Figure \ref{fig:local:moves}).

Let us consider  a  horizontal  path $\tau_h$ as in Figure \ref{fig:pillow}
and a lift of $\tau_h$ to
the surface
$M_N(a_1,a_2,a_3,a_4)$. The end-point of the resulting
path  is  the  image of the action of $T^{a_1+a_4}=T^{-(a_2+a_3)}$ on
the starting point of the lifted path. Hence, ``moving two squares to
the  right''  on  $M_N(a_1,a_2,a_3,a_4)$  we move from a square number $j$ to a square
number  $j+2(a_1+a_4)\pmod{2N}$  if  vertices $B$ and $C$ are at the
bottom    of    the    squares,    and    to    a    square    number
$j-2(a_1+a_4)\pmod{2N}$  if  vertices  $B$ and $C$ are on top of the
squares (see Figure \ref{fig:cartoon}).

\begin{figure}[htb]
\includegraphics{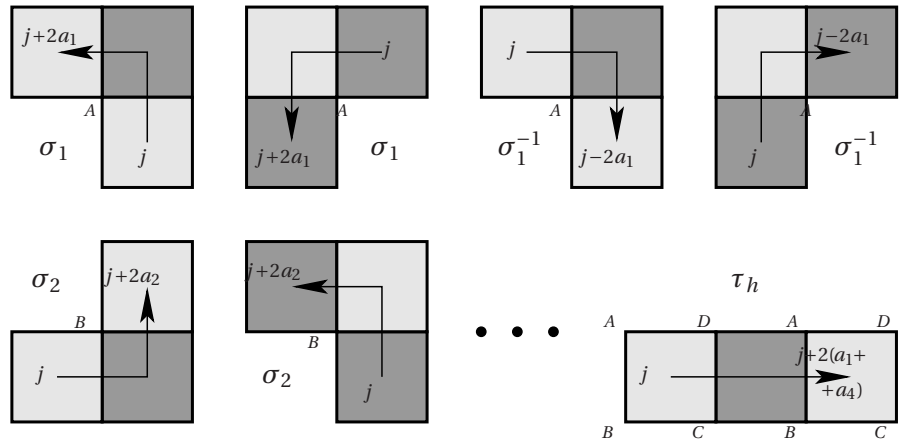}
\begin{picture}(0,0)(100,0)
\begin{picture}(0,0)(0,0)
\put(-55,-63){$\sigma_1$}
\put(-38,-48){\tiny\textit A}
\put(-18,-66){\scriptsize $j$}
\put(-62,-20){\scriptsize $j\!+\!2a_1$}
\end{picture}
\begin{picture}(0,0)(-100,0)
\put(-32,-63){$\sigma_1$}
\put(-45,-48){\tiny\textit A}
\put(-76,-66){\scriptsize $j\!+\!2a_1$}
\put(-28,-25){\scriptsize $j$}
\end{picture}
\begin{picture}(0,0)(-171,0)
\put(-57,-63){$\sigma^{-1}_1$}
\put(-37.5,-48){\tiny\textit A}
\put(-27,-66){\scriptsize $j\!-\!2a_1$}
\put(-54,-25){\scriptsize $j$}
\end{picture}
\begin{picture}(0,0)(-271,0)
\put(-32,-63){$\sigma^{-1}_1$}
\put(-45,-48){\tiny\textit A}
\put(-65,-66){\scriptsize $j$}
\put(-40,-20){\scriptsize $j\!-\!2a_1$}
\end{picture}
\begin{picture}(0,0)(0,88)
\put(-67,-25){$\sigma_2$}
\put(-51,-40.5){\tiny\textit B}
\put(-66,-60){\scriptsize $j$}
\put(-40,-23){\scriptsize $j\!+\!2a_2$}
\end{picture}
\begin{picture}(0,0)(-100,89)
\put(-82,-60){$\sigma_2$}
\put(-65,-47){\tiny\textit B}
\put(-44,-66){\scriptsize $j$}
\put(-89,-20){\scriptsize $j\!+\!2a_2$}
\end{picture}
\begin{picture}(0,0)(-200,88)
\put(-7,-25){$\tau_h$}
\put(-55,-40){\tiny\textit A}
\put(-20,-40){\tiny\textit D}
\put(14,-40){\tiny\textit A}
\put(48,-40){\tiny\textit D}
\put(-56,-82){\tiny\textit B}
\put(-22,-82){\tiny\textit C}
\put(13,-82){\tiny\textit B}
\put(47,-82){\tiny\textit C}
\put(-42,-60){\scriptsize $j$}
\put(17,-53){\scriptsize $j\!\!+\!2(\!a_1\!+$}
\put(27,-66){\scriptsize $+a_4)$}
\end{picture}
\end{picture}
\vspace{180bp}
\caption{
\label{fig:local:moves}
Local moves on a square-tiled cyclic cover}
\end{figure}

Using  these  rules, it is easy to determine the permutations $\pi_h$
and  $\pi_v$.  Start with two neighboring squares numbered by $0$ and
$1$.  By  iterating  the operation $\tau_h$, we can determine all the
squares  to  the  right  of  $0$  and  $1$ till we close up and get a
cylinder.  Recall  that we associate letters $A,B,C,D$ to the corners
of the squares. By applying appropriate operations $\sigma_i$ we find
the  squares  located atop of those which are already constructed. By
applying  appropriate  operations  $\sigma^{-1}_i$  we  find a direct
neighbor  to  the right for every square which does not belong to one
of  the  previously  constructed  horizontal  cylinders.  Having  two
horizontally  adjacent  squares,  we  apply iteratively the operation
$\tau_h$   to   obtain   all   $2N/\gcd(N,a_1+a_4)$  squares  in  the
corresponding cylinder (row), etc.

\begin{figure}[htb]
\includegraphics{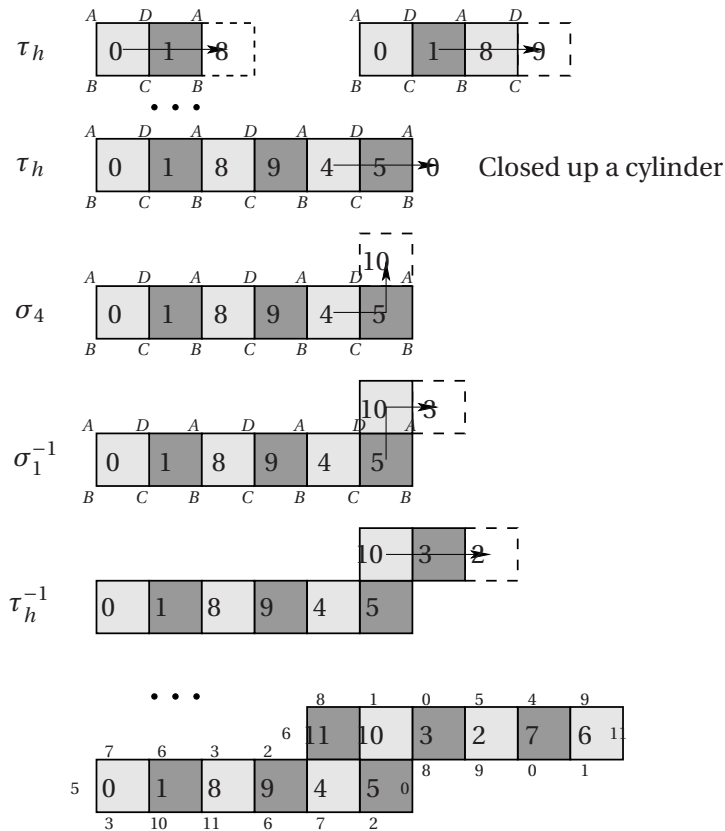}
\begin{picture}(0,0)(100,0)
\begin{picture}(0,0)(0,0.5)
\put(-25,-23){$\tau_h$}
\put(10,-25){0}
\put(30,-25){1}
\put(50,-25){8}
\put(110,-25){0}
\put(130,-25){1}
\put(150,-25){8}
\put(170,-25){9}
\put(1,-10){\tiny\textit A}
\put(21,-10){\tiny\textit D}
\put(41,-10){\tiny\textit A}
\put(1,-37){\tiny\textit B}
\put(21,-37){\tiny\textit C}
\put(41,-37){\tiny\textit B}
\put(101,-10){\tiny\textit A}
\put(121,-10){\tiny\textit D}
\put(141,-10){\tiny\textit A}
\put(161,-10){\tiny\textit D}
\put(101,-37){\tiny\textit B}
\put(121,-37){\tiny\textit C}
\put(141,-37){\tiny\textit B}
\put(161,-37){\tiny\textit C}
\end{picture}
\begin{picture}(0,0)(2.5,44.5) 
\put(-25,-23){$\tau_h$}
\put(10,-25){0}
\put(30,-25){1}
\put(50,-25){8}
\put(70,-25){9}
\put(90,-25){4}
\put(110,-25){5}
\put(130,-25){0}
\put(150,-25){Closed up a cylinder}
\put(1,-10){\tiny\textit A}
\put(21,-10){\tiny\textit D}
\put(41,-10){\tiny\textit A}
\put(61,-10){\tiny\textit D}
\put(81,-10){\tiny\textit A}
\put(101,-10){\tiny\textit D}
\put(121,-10){\tiny\textit A}
\put(1,-37){\tiny\textit B}
\put(21,-37){\tiny\textit C}
\put(41,-37){\tiny\textit B}
\put(61,-37){\tiny\textit C}
\put(81,-37){\tiny\textit B}
\put(101,-37){\tiny\textit C}
\put(121,-37){\tiny\textit B}
\end{picture}
\begin{picture}(0,0)(5,100)
\put(-25,-23){$\sigma_4$}
\put(10,-25){0}
\put(30,-25){1}
\put(50,-25){8}
\put(70,-25){9}
\put(90,-25){4}
\put(110,-25){5}
\put(106,-5){10}
\put(1,-10){\tiny\textit A}
\put(21,-10){\tiny\textit D}
\put(41,-10){\tiny\textit A}
\put(61,-10){\tiny\textit D}
\put(81,-10){\tiny\textit A}
\put(101,-10){\tiny\textit D}
\put(121,-10){\tiny\textit A}
\put(1,-37){\tiny\textit B}
\put(21,-37){\tiny\textit C}
\put(41,-37){\tiny\textit B}
\put(61,-37){\tiny\textit C}
\put(81,-37){\tiny\textit B}
\put(101,-37){\tiny\textit C}
\put(121,-37){\tiny\textit B}
\end{picture}
\begin{picture}(0,0)(8,156) 
\put(-25,-23){$\sigma^{-1}_1$}
\put(10,-25){0}
\put(30,-25){1}
\put(50,-25){8}
\put(70,-25){9}
\put(90,-25){4}
\put(110,-25){5}
\put(106,-5){10}
\put(130,-5){3}
\put(1,-10){\tiny\textit A}
\put(21,-10){\tiny\textit D}
\put(41,-10){\tiny\textit A}
\put(61,-10){\tiny\textit D}
\put(81,-10){\tiny\textit A}
\put(103,-10){\tiny\textit D}
\put(123,-10){\tiny\textit A}
\put(1,-37){\tiny\textit B}
\put(21,-37){\tiny\textit C}
\put(41,-37){\tiny\textit B}
\put(61,-37){\tiny\textit C}
\put(81,-37){\tiny\textit B}
\put(101,-37){\tiny\textit C}
\put(121,-37){\tiny\textit B}
\end{picture}
\begin{picture}(0,0)(12,211) 
\put(-25,-23){$\tau_h^{-1}$}
\put(10,-25){0}
\put(30,-25){1}
\put(50,-25){8}
\put(70,-25){9}
\put(90,-25){4}
\put(110,-25){5}
\put(106,-5){10}
\put(130,-5){3}
\put(150,-5){2}
\end{picture}
\begin{picture}(0,0)(14,279.5) 
\put(10,-25){0}
\put(30,-25){1}
\put(50,-25){8}
\put(70,-25){9}
\put(90,-25){4}
\put(110,-25){5}
\put(86,-5){11}
\put(106,-5){10}
\put(130,-5){3}
\put(150,-5){2}
\put(170,-5){7}
\put(190,-5){6}
\put(-2,-24){\tiny 5}
\put(11,-10){\tiny 7}
\put(31,-10){\tiny 6}
\put(51,-10){\tiny 3}
\put(71,-10){\tiny 2}
\put(78,-3){\tiny 6}
\put(91,10){\tiny 8}
\put(11,-37){\tiny 3}
\put(28,-37){\tiny 10}
\put(48,-37){\tiny 11}
\put(71,-37){\tiny 6}
\put(91,-37){\tiny 7}
\put(111,-37){\tiny 2}
\put(123,-24){\tiny 0}
\put(111,10){\tiny 1}
\put(131,10){\tiny 0}
\put(151,10){\tiny 5}
\put(171,10){\tiny 4}
\put(191,10){\tiny 9}
\put(131,-17){\tiny 8}
\put(151,-17){\tiny 9}
\put(171,-17){\tiny 0}
\put(191,-17){\tiny 1}
\put(202,-3){\tiny 11}
\end{picture}
\end{picture}
\vspace{320bp} 
\caption{
\label{fig:cartoon}
Cartoon movie construction of $M_6(1,1,1,3)$.}
\end{figure}

\begin{example}
Figure \ref{fig:cartoon}  presents   a   construction   of   the
enumeration for the square-tiling of $M_6(1,1,1,3)$, where the exponents
$\{1,1,1,3\}$    are    represented    by    vertices   $\{A,B,C,D\}$
respectively.  Note that by moving two squares to the right in the
first  row  (say,  $0\longrightarrow  8$) we apply $\tau_h$, while by
moving   two   squares   to   the  right  in  the  second  row  (say,
$10\longrightarrow  2$)  we  apply $\tau^{-1}_h$. In this example the
permutations  $\pi_h$  and  $\pi_v$ have the following decompositions
into cycles
\begin{align*}
\pi_h&=(0,1,8,9,4,5)(11,10,3,2,7,6)\\
\pi_v&=(0,7,4,11,8,3)(1,6,9,2,5,10)
\end{align*}
\end{example}

When  $N$  is  odd,  or when $N$ is even but at least one of $a_i$ is
also  even,  the  quadratic  differential $q=p^\ast q_0$ on $M_N(a_1,a_2,a_3,a_4)$ is
\textit{not}  a global square of a holomorphic $1$-form. The holonomy
of  the  flat  structure  defined  by  $q$  is no longer trivial: a
parallel transport of a tangent vector $v$ along certain closed paths
brings  it  to  $-v$.  In  particular  the  notions  of ``up--down'' or
``left--right''  are  no longer globally  defined.  However, the
notions  of  horizontal direction and of vertical direction are still
globally  well-defined.  Our  flat  surface  is  square-tiled and its
combinatorial geometry can still be encoded by a pair of permutations
$\pi_h$ and $\pi_v$.

Note  that  the  vertices of every square of our tiling are naturally
labeled   by   indices   $A,B,C,D$  according to the label of their projections
on the Riemann sphere. By convention, $\pi_h(2k)$ indicates
the number of a black square adjacent to the side $[CD]$ of the white
square  $2k$, and $\pi_h(2k+1)$ indicates the number of a white square
adjacent  to  the  side $[AB]$ of the black square $2k+1$. Similarly,
$\pi_v(2k)$  indicates  the  number of a black square adjacent to the
side  $[AD]$ of the white square $2k$, and $\pi_h(2k+1)$ indicates the
number  of  a  white  square adjacent to the side $[BC]$ of the black
square $2k+1$.

With this definition the remaining part of the construction literally
coincides  with  the  one  for  an  Abelian  differential,  which was
described above.

\begin{figure}[htb]
\includegraphics{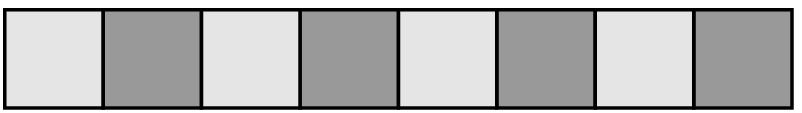}
\includegraphics{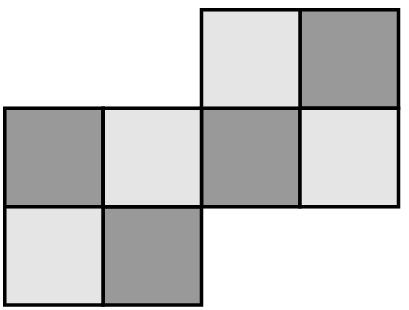}
\begin{picture}(0,0)(150,0) 
\begin{picture}(0,0)(1,23.5) 
\put(10,-25){0}
\put(30,-25){1}
\put(50,-25){6}
\put(70,-25){7}
\put(90,-25){4}
\put(110,-25){5}
\put(130,-25){2}
\put(150,-25){3}
\put(1,-10){\tiny\textit A}
\put(21,-10){\tiny\textit D}
\put(41,-10){\tiny\textit A}
\put(61,-10){\tiny\textit D}
\put(81,-10){\tiny\textit A}
\put(101,-10){\tiny\textit D}
\put(121,-10){\tiny\textit A}
\put(141,-10){\tiny\textit D}
\put(161,-10){\tiny\textit A}
\put(1,-37){\tiny\textit B}
\put(21,-37){\tiny\textit C}
\put(41,-37){\tiny\textit B}
\put(61,-37){\tiny\textit C}
\put(81,-37){\tiny\textit B}
\put(101,-37){\tiny\textit C}
\put(121,-37){\tiny\textit B}
\put(141,-37){\tiny\textit C}
\put(161,-37){\tiny\textit B}
\put(-2,-24){\tiny 3}
\put(11,-10){\tiny 5}
\put(31,-10){\tiny 4}
\put(51,-10){\tiny 3}
\put(71,-10){\tiny 2}
\put(91,-10){\tiny 1}
\put(111,-10){\tiny 0}
\put(131,-10){\tiny 7}
\put(151,-10){\tiny 6}
\put(164,-24){\tiny 0}
\put(11,-37){\tiny 5}
\put(31,-37){\tiny 4}
\put(51,-37){\tiny 3}
\put(71,-37){\tiny 2}
\put(91,-37){\tiny 1}
\put(111,-37){\tiny 0}
\put(131,-37){\tiny 7}
\put(151,-37){\tiny 6}
\end{picture}
\begin{picture}(0,0)(-217,23.5) 
\put(10,-45){0}
\put(30,-45){1}
\put(10,-25){5}
\put(30,-25){4}
\put(50,-25){7}
\put(70,-25){6}
\put(50,-5){2}
\put(70,-5){3}
\put(38,10){\tiny\textit A}
\put(61,10){\tiny\textit D}
\put(84,10){\tiny\textit A}
\put(-2,-10){\tiny\textit B}
\put(21,-10){\tiny\textit C}
\put(84,-14){\tiny\textit B}
\put(-2,-33){\tiny\textit A}
\put(61,-37){\tiny\textit D}
\put(84,-37){\tiny\textit A}
\put(-2,-57){\tiny\textit B}
\put(21,-57){\tiny\textit C}
\put(44,-57){\tiny\textit B}
\put(-2,-44){\tiny 3}
\put(-2,-24){\tiny 2}
\put(44,-44){\tiny 6}
\put(84,-24){\tiny 1}
\put(84,-4){\tiny 0}
\put(38,-4){\tiny 5}
\put(11,-57){\tiny 5}
\put(31,-57){\tiny 4}
\put(51,-37){\tiny 2}
\put(71,-37){\tiny 3}
\put(11,-10){\tiny 0}
\put(31,-10){\tiny 1}
\put(51,10){\tiny 7}
\put(71,10){\tiny 6}
\end{picture}
\end{picture}
\vspace{85pt}
\caption{
\label{fig:nonorientable}
These  two  ways  to  unfold  the square-tiled
cyclic cover $M_4(1,3,2,2)$ represent
decompositions into horizontal and into vertical cylinders.
}
\end{figure}

\begin{example}
Figure \ref{fig:nonorientable}   illustrates   a   square-tiling   of
$M_4(1,3,2,2)$.  The  flat metric has nontrivial linear holonomy; the
corresponding   quadratic   differential   belongs   to  the  stratum
$\mathcal{Q}(2,2)$.  In this example the permutations $\pi_h$ and $\pi_v$ are
decomposed into cycles as
\begin{align*}
\pi_h&=(0,1,6,7,4,5,2,3)\\
\pi_v&=(0,5)(1,4)(2,7)(3,6)
\end{align*}
\end{example}

\begin{remark}
Note  that  when  pairing  sides  of  boundary squares of an abstract
square-tiled  surface  one  has  to  respect  the  orientation of the
surface.
\end{remark}

We  proceed  below  with  an elementary lemma which will, however, be
important later.

\begin{lemma}
\label{lm:cylinder:width}
Consider a decomposition $M_N(a_1,a_2,a_3,a_4)= \sqcup{\rm cyl}_i $ of a square-tiled
cyclic cover into cylinders ${\rm cyl}_i$ filled by closed horizontal
trajectories. For every cylinder ${\rm cyl}_i$ we denote by $w_i$ its
width  (the length of each closed horizontal trajectory) and by $h_i$
its height (the length of each vertical segment).

Assuming  that the branch points $z_1,z_2,z_3,z_4$ are numbered as indicated
in Figure \ref{fig:pillow}, the widths of the corresponding cylinders
and  the  sum  of  the  heights  of  all  cylinders  are given by the
formulas:
$$
w_i=  \cfrac{2N}{\gcd(N,a_1+a_4)}    \,,\ \ \text{ for all } i \,, \text{ and } \,\,
\sum_i h_i = \gcd(N,a_1+a_4)\,.
$$
\end{lemma}
\begin{proof}
Clearly,    the    operation    $\tau_h=T^{(a_1+a_4)}$    has   order
$\tfrac{N}{\gcd(N,a_1+a_4)}$.  Hence,  the  length of each horizontal
trajectory  is  equal to $\frac{2N}{\gcd(N,a_1+a_4)}$ which, in turn,
is  equal to the width of any cylinder. Since the area of the surface
is   $2N$,   the   total   height   of  all  cylinders  is  equal  to
$\gcd(N,a_1+a_4)$.
\end{proof}

\begin{remark}
It  is irrelevant  whether or not the  quadratic  differential
$q=p^\ast   q_0$   defining   the   square-tiled  flat  structure  in
Lemma \ref{lm:cylinder:width}  is  a  global  square of a holomorphic
$1$-form, or not.
\end{remark}

\subsection{Symmetries of cyclic covers}
\label{sec:symmetries}

We  continue  with a description of the isomorphisms of cyclic covers
(certainly known to all who worked with them). Note that in the Lemma
below we do not use any flat structure.

\begin{lemma}
\label{lm:dual}
Two    cyclic    covers    $M_N(a_1,\dots,a_4)$    and    $M_N(\tilde
a_1,\dots,\tilde  a_4)$  admit  an  isomorphism  compatible  with the
projection to $\mathbb{P}^1(\mathbb{C})$

\begin{picture}(80,70)(-80,-50)
\put(0,0){$M_N(a_1,\dots,a_4)\simeq
M_N(\tilde a_1,\dots,\tilde a_4)$}
\put(40,-10){\vector(1,-1){20}}
\put(130,-10){\vector(-1,-1){20}}
\put(70,-40){$\mathbb{P}^1(\mathbb{C})$}
\end{picture}

\noindent
if and only if for some primitive element $k\in\mathbb{Z}/N\mathbb{Z}$ one has
\begin{equation}
\label{eq:mult:by:k}
\tilde a_i=k a_i\pmod N,\quad\text{ for }\quad i=1,2,3,4\ .
\end{equation}
In particular,
\begin{equation}
\label{eq:dual}
M_N(a_1,a_2,a_3,a_4)\simeq M_N(N-a_1,N-a_2,N-a_3,N-a_4)\ .
\end{equation}
\end{lemma}
\begin{proof}
Any  isomorphism $g$ of cyclic covers as above induces an isomorphism
$g_\ast$ of their groups of deck transformations
such that,
$$
\widetilde{\operatorname{Deck}}(\sigma_i)=g_\ast(\operatorname{Deck}(\sigma_i)) \, , \quad \text{ \rm for }i=1, \dots, 4 \,.
$$
  Since  by construction $\widetilde{\operatorname{Deck}}(\sigma_i)=\tilde  a_i$
and $\operatorname{Deck}(\sigma_i)=a_i$ in $\mathbb{Z}/N\mathbb{Z}$, for $i=1, \dots, 4$,  and any  automorphism of a cyclic
group $\mathbb{Z}/N\mathbb{Z}$ is given by the multiplication by a primitive element $k\in\mathbb{Z}/N\mathbb{Z}$,  the above relation yields formula \eqref{eq:mult:by:k}.

On  the other hand, when condition \eqref{eq:mult:by:k} is satisfied,
one can find
integers $m_1, \dots, m_4$ such that $\tilde a_i = ka_i +
m_i  N$,  where  $0<\tilde  a_i\le N$, and, hence, we have an obvious
isomorphism
$$
\tilde w= w^k (z-z_1)^{m_1}  (z-z_2)^{m_2} (z-z_3)^{m_3} (z-z_4)^{m_4}
$$
between                the                cyclic               covers
$w^N=(z-z_1)^{a_1}(z-z_2)^{a_2}(z-z_3)^{a_3}(z-z_4)^{a_4}$        and
${\tilde w}^N=(z-z_1)^{\tilde a_1}(z-z_2)^{\tilde a_2}(z-z_3)^{\tilde
a_3}(z-z_4)^{\tilde a_4}$.
\end{proof}

Consider  a  particular  case  when $\{a_1,\dots,a_4\}$ and $\{\tilde
a_1,\dots,\tilde  a_4\}$  coincide  as  unordered sets (possibly with
multiplicities).

\begin{definition}
\label{def:symmetry}
A permutation $\pi$ in $\mathfrak{S}_4$ is called a \textit{symmetry}
of a cyclic cover $M_N(a_1,a_2,a_3,a_4)$ if there exists an integer $k$ such that
$$
k\cdot a_i\pmod{N}=a_{\pi(i)}\quad\text{for }\ i=1,2,3,4\,.
$$
\end{definition}

\subsection{Veech group of a square-tiled cyclic cover}
\label{ss:Teichmuller:curve:associated:to:a:cyclic:cover}

The  group  $\textrm{SL}(2,\mathbb{R})$ and $\textrm{PSL}(2,\mathbb{R})$  acts  naturally  on any
stratum   of   holomorphic   $1$-forms and, respectively, meromorphic
quadratic differentials with at most simple poles. The \textit{Veech
group}  $\Gamma(S)$  of  a  flat surface $S$ is the stabilizer of the
corresponding point of the stratum under this action. In this section
we   study  the  Veech  groups  of  square-tiled  cyclic  covers.  In
particular, we prove the following statement.

\begin{theorem}
\label{th:SLZ:orbit}
The  Veech  group $\Gamma(S)$
of any square-tiled cyclic cover $S$
contains the group $\Gamma(2)$
(respectively $\Gamma(2)/(\pm\operatorname{Id})$)
as a subgroup.
The  Veech  group $\Gamma(S)$
has one of the indices $1$, $2$,
$3$ or $6$ in $\textrm{SL}(2,\mathbb{Z})$ (respectively in $\textrm{PSL}(2,\mathbb{Z})$).

If the Veech groups $\Gamma(S_1)$
and $\Gamma(S_2)$ of two square-tiled cyclic covers
$S_1, S_2$ have the same index in $\textrm{SL}(2,\mathbb{Z})$ (respectively in $\textrm{PSL}(2,\mathbb{Z})$),
then $\Gamma(S_1)$ and $\Gamma(S_2)$ are conjugate.
\end{theorem}
\begin{remark}A proper subgroup of $\textrm{SL}(2,\mathbb{Z})$ containing $\Gamma(2)$ is said to be a \emph{congruence subgroup of level two}. It is not hard to check that the conjugation class of a congruence subgroup of level two is uniquely determined by its index. An example of a congruence subgroup of $\textrm{SL}(2,\mathbb{Z})$ of level two with index $3$ is
$$\Gamma_0(2)=\left\{\left(\begin{array}{cc}a & b \\ c & d\end{array}\right): a\equiv d\equiv 1 (\textrm{mod } 2), c\equiv 0 (\textrm{mod }2)\right\}.$$In the literature, the congruence subgroups of level two of $\textrm{SL}(2,\mathbb{Z})$ are called \emph{theta groups} (sometimes denoted $\Theta$).
\end{remark}

Theorem \ref{th:SLZ:orbit}  will  be  derived  from
Lemma \ref{l.Lemma2-5} and from a  more  precise Theorem \ref{thm:1prime}.

Sometimes  it  is convenient to ``mark'' (in other words ``label'' or
``give  names  to'')  the  zeroes  (and  the  simple  poles)  of  the
corresponding   holomorphic   $1$-form   (respectively,   quadratic
differential).  In  this way, one gets the strata of marked (in other
words ``labeled'', ``named''), flat surfaces. The action of the group
$\textrm{SL}(2,\mathbb{R})$ (respectively $\textrm{PSL}(2,\mathbb{R})$) on the strata of marked flat surfaces,
and  the  Veech  group $\Gamma(S_{\mathit{marked}})$ of a marked flat
surface $S_{\mathit{marked}}$ are defined analogously.

Let  the  flat  surface $S$ be a ``unit square pillow'', that is, let
$S$      be      $\mathbb{P}^1(\mathbb{C})$      endowed      with      the      quadratic
differential \eqref{eq:q:on:CP1},  where the parameters are chosen in
such  way  that  the flat surface $S$ is glued from two unit squares,
and  their  sides  are  vertical  and  horizontal.  The  Veech  group
$\Gamma(S)$ of such a flat surface $S$ coincides with $\textrm{PSL}(2,\mathbb{Z})$. We will
need  the  following  elementary  Lemma for the marked version of the
latter surface.

\begin{lemma}\label{l.Lemma2-5}
The  action  of  the  group $\textrm{PSL}(2,\mathbb{Z})$ on the ``unit square pillow with
marked  corners''  factors  through  the  free  action of the group
$\mathfrak{S}_3$  of  permutations  of  three  elements  by  means of the
surjective homomorphism
\begin{equation}
\label{eq:homomorphism:to:symmetric:group}
\textrm{PSL}(2,\mathbb{Z})\to\operatorname{PSL}(2,\mathbb{Z}/2\mathbb{Z})\simeq\mathfrak{S}_3\,.
\end{equation}
In  particular,  the Veech group $\Gamma(S_{\mathit{marked}})$ of the
``unit  square  pillow  with  marked  corners'' $S_{\mathit{marked}}$
coincides with the kernel $\Gamma(2)/(\pm\operatorname{Id})$ of this homomorphism.
\end{lemma}
\begin{proof}
Let  the labels of the corners of the pillow be $A, B, C, D$, say, as
in  Figure \ref{fig:pillow}.  An  element  $g$ of $\textrm{PSL}(2,\mathbb{Z})$ acts on the
flat  surface  $(\mathbb{P}^1(\mathbb{C}),  q_0)$, giving a new flat surface isomorphic to
the original one. It can still be obtained by gluing two squares, but
the  labels  $A,  B,  C,  D$  have moved around.
\begin{figure}[hbt]
\includegraphics{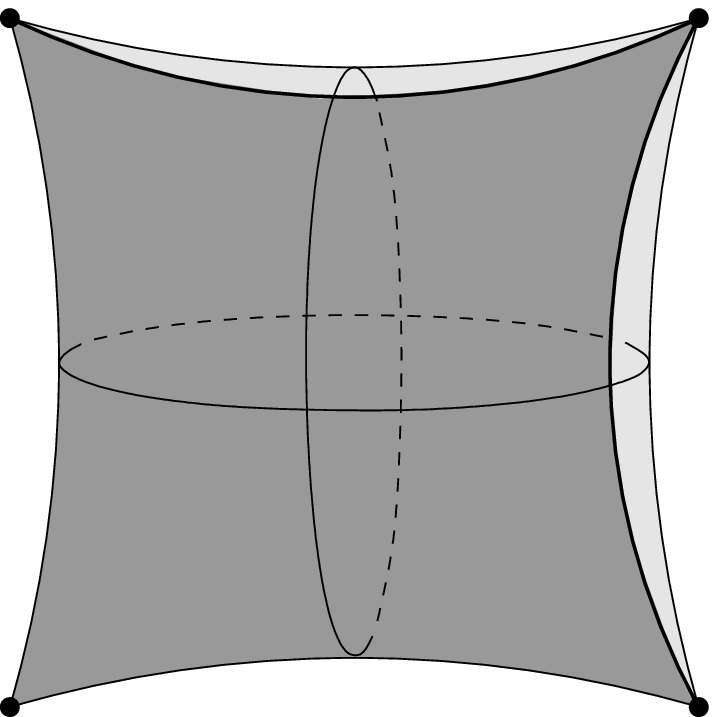}
\includegraphics{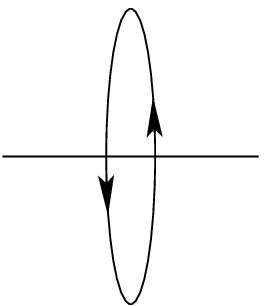}
\includegraphics{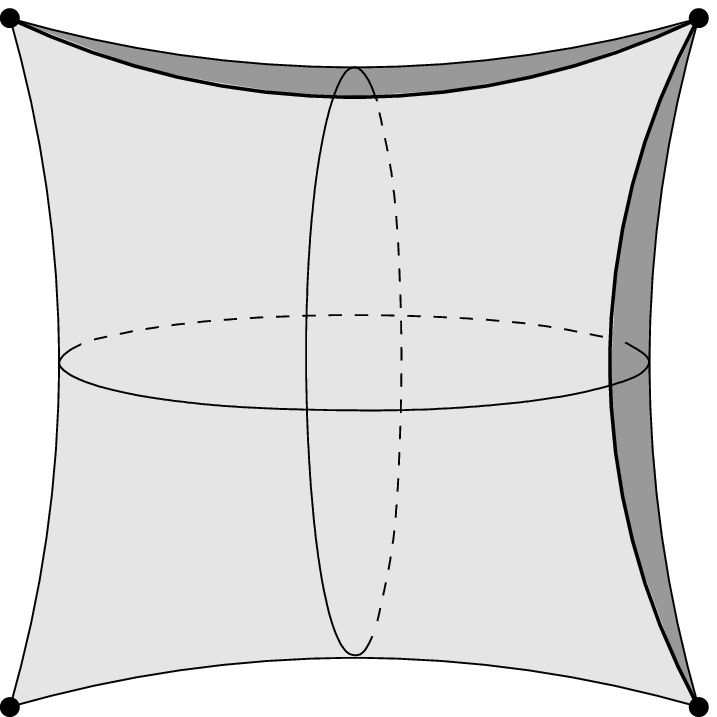}
\includegraphics{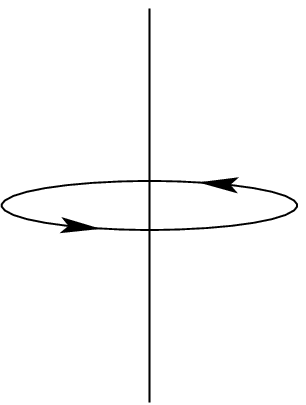}
\includegraphics{pillow_black.eps}
\begin{picture}(0,0)(0,0)
\begin{picture}(0,0)(225,0)
\begin{picture}(0,0)(0,0)
\put(61,-3){$z_1$}
\put(61,-71){$z_2$}
\put(134,-71){$z_3$}
\put(134,-3){$z_4$}
\put(76,-19){$B$}
\put(76,-60){$A$}
\put(116,-60){$D$}
\put(116,-19){$C$}
\end{picture}
\end{picture}
\begin{picture}(0,0)(102,0)
\begin{picture}(0,0)(0,0)
\put(61,-3){$z_1$}
\put(61,-71){$z_2$}
\put(134,-71){$z_3$}
\put(134,-3){$z_4$}
\put(76,-19){$A$}
\put(76,-60){$B$}
\put(116,-60){$C$}
\put(116,-19){$D$}
\end{picture}
\end{picture}
\begin{picture}(0,0)(-20,0)
\begin{picture}(0,0)(0,0)
\put(61,-3){$z_1$}
\put(61,-71){$z_2$}
\put(134,-71){$z_3$}
\put(134,-3){$z_4$}
\put(76,-19){$D$}
\put(76,-60){$C$}
\put(116,-60){$B$}
\put(116,-19){$A$}
\end{picture}
\end{picture}
\end{picture}
\vspace{70bp}
\caption{
\label{fig:pillow:symmetries}
Symmetries of the flat surface $\left(\mathbb{P}^1(\mathbb{C}),q_0\right)$.
}
\end{figure}
Using an appropriate
``pillow symmetry'' as in Figure \ref{fig:pillow:symmetries}, one can
always  move  back one chosen label to the original position, but the
other  labels  $B,C,D$  are  not  fixed.  Speaking more formally, the
``pillow symmetries'' define the normal subgroup
\begin{equation}
\label{eq:Klein:subgroup}
\mathfrak{K}=\big\{ ()\,;(1,2)(3,4)\,;(1,3)(2,4)\,;(1,4)(2,3)\big\}
\end{equation}
of  the symmetric group $\mathfrak{S}_4$. The subgroup $\mathfrak{K}$
is  isomorphic  to  the Klein group. The quotient of $\mathfrak{S}_4$
over   $\mathfrak{K}$   is   isomorphic   to   the   symmetric  group
$\mathfrak{S}_3$. The projection
\begin{equation}
\label{eq:projection:S4:to:S3}
\mathfrak{S}_4\to\mathfrak{S}_3\simeq\mathfrak{S}_4/\mathfrak{K}
\end{equation}
is  not  canonical,  since it depends on the choice of the fixed label.
However, the \textit{conjugacy class} of the image of any element is defined canonically.

Note  that  the ``pillow symmetries'' are diffeomorphisms of our flat
sphere with  differential  $\pm\operatorname{Id}$  in flat coordinates.
Thus,  all  ``unit  square  pillows  with marked corners'' related by
``pillow  symmetries''  define  one and the same point of the stratum
$\mathcal{Q}_{\mathit{marked}}(-1,-1,-1,-1)$.

It  is easy to see that under the convention that $\textrm{PSL}(2,\mathbb{Z})$ keeps
one of the labels fixed, the elements $\begin{pmatrix}1&2\\0&1\end{pmatrix}$
and $\begin{pmatrix}1&0\\2&1\end{pmatrix}$  of $\textrm{PSL}(2,\mathbb{Z})$ fix \text{all} the
labels,    and,    thus,    belong to the Veech group of
$S_{\mathit{marked}}$.  It  is  a  well-known fact  that the above two
elements    generate   the   kernel   $\Gamma(2)/(\pm\operatorname{Id})$   of   the
homomorphism of formula \eqref{eq:homomorphism:to:symmetric:group},
and that the group $\operatorname{PSL}(2,\mathbb{Z}/2\mathbb{Z})=\operatorname{SL}(2,\mathbb{Z}/2\mathbb{Z})$ is
isomorphic to $\mathfrak{S}_3$.

It  is  also  easy  to  check that under this identification the
elements of $\mathfrak{S}_3$ act by the corresponding permutations of the
three  ``free''  labels; in particular, the action of $\mathfrak{S}_3$ on
the six distinct ``unit square pillows with marked corners'' is free.
\end{proof}

Let  $S$  be  a  cyclic  cover of a type $M_N(a_1,\dots,a_4)$ endowed with the flat
structure  induced  from  the  flat  structure \eqref{eq:q:on:CP1} on
$\mathbb{P}^1(\mathbb{C})$.  As  always, we assume that when $N$ is even and all $a_i$ are
odd,  the  flat structure on $S$ is defined by the holomorphic $1$-form
$\omega$, such that $\omega^2=p^\ast q_0$; otherwise it is defined by
the quadratic differential $q=p^\ast q_0$.

It  is  easy  to  see,  that  for  any $g$ in $\textrm{SL}(2,\mathbb{R})$ (respectively, for any
$g$ in $\textrm{PSL}(2,\mathbb{R})$), the image $\tilde S:=gS$ is also represented by a cyclic
cover. Moreover, any affine diffeomorphism $A_g:S\to \tilde S$ of the
corresponding  flat  surfaces  intertwines the action of the group of
the  deck  transformations  on $S$ and $\tilde S$, that is, $A_g\circ
T=\tilde  T\circ A_g$. Similarly, for any $g\in\textrm{SL}(2,\mathbb{Z})$ (respectively, for any
$g\in\textrm{PSL}(2,\mathbb{Z})$)  the image of a \textit{square-tiled} cyclic cover under
the action of $g$ is again a \textit{square-tiled} cyclic cover.

In the rest of this section (and, basically, in the remaining part of
the  paper)  we consider only \textit{square-tiled} cyclic covers, in
particular,  to  avoid  cumbersome  notations,  we denote by $M_N(a_1,a_2,a_3,a_4)$ a
cyclic  cover  endowed  with  the  flat  structure  induced  from the
``square   pillow''   as   in   Figure \ref{fig:pillow}.  Under  this  convention,
the order of the entries $a_1,\dots,a_4$ matters in the definition of a  square-tiled  cyclic
cover  $M_N(a_1,a_2,a_3,a_4)$.  However, the square-tiled cyclic covers
$$
M_N(a_1,a_2,a_3,a_4), M_N(a_2,a_1,a_4,a_3),
M_N(a_4,a_3,a_2,a_1), \text{ and } M_N(a_3,a_4,a_1,a_2)\,,
$$
related  by ``pillow symmetries''  (see formula \eqref{eq:Klein:subgroup})
define the same flat surface.

The  fact  that  the square-tiled cyclic cover $M_N(a_3,a_4,a_1,a_2)$
defines  the same flat surface as $M_N(a_1,a_2,a_3,a_4)$ implies, in particular, that
in the case when the flat structure on a square-tiled cyclic cover is
defined      by     a     holomorphic     $1$-form,     the     element
$\begin{pmatrix}-1&0\\0&-1\end{pmatrix}$  of  $\textrm{SL}(2,\mathbb{Z})$  belongs  to the
Veech  group  of the corresponding flat surface. Hence, the action of
$\textrm{SL}(2,\mathbb{Z})$  on square-tiled cyclic covers factors through the action of
$\textrm{PSL}(2,\mathbb{Z})$.  We  shall  sometimes  consider the latter action
without specifying it explicitly.

It  is easy to see, that if an element $g\in\textrm{PSL}(2,\mathbb{Z})$ permutes the
marking  of the initial ``unit square pillow with marked corners'' by
a  permutation  $\pi\in\mathfrak{S}_4$,  then the square-tiled cyclic
cover  $M_N(a_1,a_2,a_3,a_4)$  is  mapped  by  $g$  to  the square-tiled cyclic cover
$M_N(a_{\pi(1)},a_{\pi(2)},a_{\pi(3)},a_{\pi(4)})$.     Since     the
elements   of   the   subgroup   $\mathfrak{K}$  of  $\mathfrak{S}_4$
correspond to isomorphic square-tiled cyclic covers, we conclude that
the  action  of  $\textrm{SL}(2,\mathbb{Z})$  (respectively of $\textrm{PSL}(2,\mathbb{Z})$) on square-tiled
cyclic  covers  factors through the action on ``unit square pillows
with  marked corners''.

By combining the latter observation with Lemma \ref{l.Lemma2-5} we
conclude that the Veech group of $M_N(a_1,a_2,a_3,a_4)$ contains the
group $\Gamma(2)$ (respectively
the group
$\Gamma(2)/(\pm\operatorname{Id})$),
that it has index \textit{at most} $6$ in $\textrm{SL}(2,\mathbb{Z})$
(respectively, in $\textrm{PSL}(2,\mathbb{Z})$) and that it is determined by
its index up to conjugation.

To complete the proof of Theorem \ref{th:SLZ:orbit}
it remains to prove that all indices $1$, $2$, $3$, $6$ are realized.

In order to  describe  this  Veech  group more precisely we need the
following remark. Let
$$
\hat f:M_N(a_1,a_2,a_3,a_4)\to M_N(\tilde a_1,\tilde a_2,\tilde a_3,\tilde a_4)
$$
be   an   isomorphism  of  square-tiled  cyclic  covers,  that  is,  a
diffeomorphism  with differential  equal to $\operatorname{Id}$  (respectively
$\pm\operatorname{Id}$) in flat coordinates. It is not hard to see that $\hat f$ is part of a commutative diagram   %
$$
\begin{CD}
M_N(a_1,a_2,a_3,a_4) @>\hat f>> M_N(\tilde a_1,\tilde a_2,\tilde a_3,\tilde a_4)\\
@VVpV @VV pV\\
(\mathbb{P}^1(\mathbb{C}),q_0) @>f>> (\mathbb{P}^1(\mathbb{C}),q_0)\,,
\end{CD}
$$
where  $p$  is  the  canonical  projection and $f:\mathbb{P}^1(\mathbb{C})\to\mathbb{P}^1(\mathbb{C})$ is an
automorphism  of  the underlying flat sphere. We have seen that the
only  automorphisms  of  the  ``square  pillow''  $(\mathbb{P}^1(\mathbb{C}),q_0)$ are the
``pillow  symmetries'' (see  Figure \ref{fig:pillow:symmetries}). It follows from
formula \eqref{eq:Klein:subgroup} that all ``pillow  symmetries'', that is, all elements
of the Klein group $\mathfrak{K}$ are involutions.  Let $s$  be  the  the ``pillow  symmetry''
corresponding to the automorphism $f:\mathbb{P}^1(\mathbb{C})\to\mathbb{P}^1(\mathbb{C})$ or, equivalently, to its inverse.
Let it act on the canonical  labeling  of  the  corners  of the pillow by a permutation
$\varkappa\in\mathfrak{K}$.  Let $\hat s$ be the induced automorphism of
$M_N(\tilde a_1,\tilde a_2,\tilde a_3,\tilde a_4)$.
By definition of $\hat s$ the diagram
$$
\begin{CD}
M_N(\tilde a_1,\tilde a_2,\tilde a_3,\tilde a_4)
@>\hat s>> M_N(\tilde a_{\kappa(1)},\tilde a_{\kappa(2)},\tilde a_{\kappa(3)},\tilde a_{\kappa(4)})
\\
@VVpV  @VV\tilde pV\\
(\mathbb{P}^1(\mathbb{C}),q_0) @>s>> (\mathbb{P}^1(\mathbb{C}),q_0)
\end{CD}
$$
commutes. Note that by construction the composition $f\circ s$ is the
identity  map,  which allows us to merge the two commutative diagrams
above into the commutative diagram

\begin{equation}
\label{eq:commut:triangle}
\begin{picture}(200,0)
\put(0,0){$M_N(a_1,\dots,a_4)\simeq M_N(\tilde a_{\kappa(1)},\tilde a_{\kappa(2)},\tilde a_{\kappa(3)},\tilde a_{\kappa(4)})$}
\put(50,-10){\vector(1,-1){20}}
\put(140,-10){\vector(-1,-1){20}}
\put(80,-40){$\mathbb{P}^1(\mathbb{C})$\,.}
\end{picture}
\vspace*{1.8cm}
\end{equation}

Let us consider  the case  when $\{a_1,\dots,a_4\}$ and $\{\tilde
a_1,\dots,\tilde  a_4\}$  coincide  as  unor\-dered sets (possibly with
multiplicities). Since all  symmetries of  cyclic  covers such as
those in formula \eqref{eq:commut:triangle}   are  described  by
Definition \ref{def:symmetry}     and     Lemma \ref{lm:dual},    our
considerations imply the following statement.

\begin{lemma}
\label{lm:symmetry}
Consider  $\pi\in\mathfrak{S}_4$.  The square-tiled  cyclic  covers
$M_N(a_1,a_2,a_3,a_4)$ and
$M_N(a_{\pi(1)},a_{\pi(2)},a_{\pi(3)},a_{\pi(4)})$   are   isomorphic
(that  is  define the same point of the corresponding stratum) if and
only  if  there  exists  a  symmetry  $\pi'$
of the cyclic cover
$M_N(a_1,a_2,a_3,a_4)$ such that the
permutation   $\pi'\cdot\pi^{-1}$   belongs  to  the  Klein  subgroup
$\mathfrak{K}$ defined in formula \eqref{eq:Klein:subgroup}.
\end{lemma}

By passing to  the quotient $\mathfrak{S}_3\simeq\mathfrak{S}_4/\mathfrak{K}$
we get the following immediate Corollary of the Lemma above.

\begin{corollary}
\label{cor:symmetries}
For any square-tiled cyclic cover $M_N(a_1,a_2,a_3,a_4)$ the index of
its Veech group in $\textrm{SL}(2,\mathbb{Z})$  (in
$\textrm{PSL}(2,\mathbb{Z})$ when the flat structure is defined by a
quadratic differential) coincides with the index of the image of the
subgroup of symmetries of $M_N(a_1,a_2,a_3,a_4)$ in $\mathfrak{S}_3$
under the projection of formula \eqref{eq:projection:S4:to:S3}.
\end{corollary}

In order to complete the proof of Theorem \ref{th:SLZ:orbit} it is
sufficient to prove that all indices $1,2,3,6$ are realized. We
prefer to prove a strengthened version of Theorem \ref{th:SLZ:orbit}.

\begin{theorem}
\label{thm:1prime}
The  index  of  the  Veech  group  of  a square-tiled cyclic cover is
described by the following list.
\begin{itemize}
\item
If for some triple of pairwise distinct indices $i,j,k\in\{1,2,3,4\}$
one has  $a_i=a_j=a_k$,  the  index  of  the  Veech  group  of  the
corresponding
square-tiled cyclic cover $M_N(a_1,a_2,a_3,a_4)$ is $1$.
\item
If  there  is  no  such a triple of pairwise distinct indices,
but there is a pair of
indices  $i\neq  j$,  where $i,j\in\{1,2,3,4\}$, such that $a_i=a_j$,
then the index of the Veech group is $3$.
\end{itemize}

If all $a_i$ are pairwise distinct, then
\begin{itemize}
\item
If  $M_N(a_1,\dots,a_4)$  does not have nontrivial symmetries, or if any nontrivial
symmetry  decomposes  into two cycles of length $2$, the index of the
Veech group is $6$.
\item
If $M_N(a_1,\dots,a_4)$ has  a symmetry represented by a cycle of length $4$ or by a
single cycle of length $2$, the index of the Veech group is $3$.
\item
If $M_N(a_1,\dots,a_4)$ has  a symmetry represented by a cycle of length $3$, the
index of the Veech group is $2$.
\end{itemize}
All the symmetries listed above are realized.
\end{theorem}
\begin{proof}
If for some triple of pairwise distinct indices $i,j,k\in\{1,2,3,4\}$
the numbers $a_i=a_j=a_k$ coincide, it is clear that the index of the
Veech  group  of the square-tiled cyclic cover $M_N(a_1,a_2,a_3,a_4)$ is $1$, say, as
for $M_4(1,1,1,1)$ or for $M_6(3,1,1,1)$.

Let us suppose  that  there  is  no such a triple of indices, but there is a
pair  of  indices  $i\neq  j$, where $i,j\in\{1,2,3,4\}$, such that
$a_i=a_j$,  say,  as  for $M_{10}(1, 1, 3, 5)$ or for $M_6(1,1,5,5)$.
Then  the transposition in the symmetric group $\mathfrak{S}_4$ which
interchanges  the  two  labels corresponding to $a_i$ and $a_j$ fixes
the  square-tiled  cyclic  cover $M_N(a_1,\dots,a_4)$. The image of a transposition
under   the   projection \eqref{eq:projection:S4:to:S3}  is  again  a
transposition.  Hence,  the  index  of the Veech group in this case is
either  $3$  or  $1$.

Let  us  show,  that  index $1$ is excluded. Let us suppose that the index of the
Veech    group    is    $1$,    that    is,   for   all   permutations
$\pi\in\mathfrak{S}_4$  the flat surfaces represented by square-tiled
cyclic  covers $M_N(a_{\pi(1)},a_{\pi(2)},a_{\pi(3)},a_{\pi(4)})$ are
isomorphic.  Thus,  without  loss  of  generality  we may assume that
$a_2=a_1$, and $a_3\neq a_1$, $a_4\neq a_1$. Then
$$
\operatorname{Deck}(\sigma_1\sigma^{-1}_2)=a_1-a_2=0\,.
$$
This  property  can  be formulated in a form invariant under ``pillow
symmetries'',  namely:  for  at  least one of the two vertical saddle
connections  of the ``square pillow'' $(\mathbb{P}^1(\mathbb{C}),q_0)$ the loop encircling
one  of  the  corresponding  singularities in positive direction, and
then  the  other  singularity in negative direction lifts to a closed
loop   on   $M_N(a_1,a_2,a_3,a_4)$.   Clearly   this   property  is
not  valid  for
the surface
$M_N(a_1,a_3,a_2,a_4)$ and we get a contradiction.

Consider now the remaining case when all $a_i$ are pairwise distinct.

If  any nontrivial symmetry decomposes into two cycles of length $2$,
all  the  symmetries  are  reduced  to  ``pillow  symmetries'' and by
Corollary \ref{cor:symmetries} the index of the Veech group of $M_N(a_1,\dots,a_4)$
is $6$. This situation realizes, for example, for $M_8(1,3,5,7)$.

Let us suppose  that  $M_N(a_1,\dots,a_4)$ has a symmetry represented by a cycle of length $4$.
As  an  example,  consider  $M_{10}(1,3,9,7)$  and  a symmetry
corresponding  to  the  multiplication  by  $k=3$.  The image of such
symmetry   under   projection \eqref{eq:projection:S4:to:S3}   is   a
transposition.  By  Corollary \ref{cor:symmetries}  this implies that
the  index of the Veech group of $M_N(a_1,\dots,a_4)$ is either $3$ or $1$. It is
an   exercise   to   verify   that   index $1$   is  excluded  (see
Appendix \ref{a:exercise}).

Let us suppose  now  that  $M_N(a_1,\dots,a_4)$  has  a symmetry represented by a cycle of
length $2$. As an example, consider $M_{40}(1,9,5,25)$ and a symmetry
induced  by  multiplication  by  $9$.  A  transposition  is mapped by the
projection \eqref{eq:projection:S4:to:S3}  to a transposition. Hence,
the  index  of  the Veech group in this case is again either $3$ or $1$.
It  is  an  exercise  to verify that index $1$ is excluded (see Appendix \ref{a:exercise}).

Finally, let us suppose  that  $M_N(a_1,\dots,a_4)$  has  a symmetry represented by a cycle of
length  $3$.  As  an  example,  consider  $M_{14}(1,9,11,7)$  and a
symmetry  induced by multiplication by $k=9$. A cycle of length $3$
is  mapped by the projection \eqref{eq:projection:S4:to:S3} to a cycle of
length $3$.  Hence,  the index of the Veech group is either $2$ or $1$.
If  it  were $1$,  one  of the symmetries would be an odd
permutation, \textit{i.e.},  a  single  cycle  of  length $2$ or $4$. We have
proved  that  the  presence  of  such  a symmetry excludes index $1$.
Theorem \ref{thm:1prime} and, thus, Theorem \ref{th:SLZ:orbit} are proved.
\end{proof}

To   complete   this   section   we   note   that   the   $\textrm{SL}(2,\mathbb{R})$-orbit
of any square-tiled surface
(respectively, the $\textrm{PSL}(2,\mathbb{R})$-orbit in the case when the flat structure
is  represented  by  a  quadratic  differential)
inside  the  ambient  moduli  space  of Abelian or quadratic
differentials is closed. Its projection to the moduli space of curves
is  a  Riemann surface with cusps, often called an \textit{arithmetic
Teichm\"uller   curve},  see \cite{Veech},  \cite{Gutkin:Judge}.  Any
arithmetic  Teichm\"uller  curve  is  a  finite  cover of the modular
curve.  Theorem \ref{th:SLZ:orbit}  shows  that  for  a  square-tiled
cyclic  cover,  the  corresponding  arithmetic Teichm\"uller curve is
very  small:  it is a $1$, $2$, $3$, or $6$-fold cover of the modular
curve.

\section{Sum of Lyapunov exponents}
\label{s:Sum:of:Lyapunov:exponents}

\subsection{Sum of the Lyapunov exponents for a square-tiled surface}
\label{ss:sum:general}

Let us consider  a  Teichm\"uller  curve  $\mathcal{C}$.  Each point $x$ of $\mathcal{C}$ is
represented   by  a  Riemann  surface  $S_x$.  We  can  consider  the
cohomology space of $H^1(S_x,\mathbb{R})$ as a fiber of a vector bundle $H^1$
over  $\mathcal{C}$,  called the \textit{Hodge bundle}. Similarly one defines
the bundles $H^{1,0}$ and $H^1_{\mathbb{C}}$. Note that each fiber is endowed
with  a  natural  integer  lattice $H^1(S_x,\mathbb{Z})$, which enables us to
identify  the  fibers  at nearby points $x_1, x_2$. Hence, the bundle
$H^1$   is  endowed  with  a  natural  flat  connection,  called  the
\textit{Gauss--Manin connection}.

The  Teichm\"uller  curve  $\mathcal{C}$ is endowed with a natural hyperbolic
metric  associated  to the complex structure of $\mathcal{C}$; the total area
of  $\mathcal{C}$  with respect to this metric is finite. Consider a geodesic
flow  in  this metric, and consider the monodromy of the Gauss--Manin
connection  in  $H^1$  with respect to the geodesic flow on $\mathcal{C}$. We
get a $2g$-dimensional symplectic cocycle. The geodesic flow on $\mathcal{C}$
is  ergodic  with  respect  to  the  natural finite Lebesgue measure.
Let us denote by $\lambda_1\ge\lambda_2\ge\dots\ge\lambda_{2g}$ the Lyapunov
exponents   of  the  corresponding  cocycle.  Since  the  cocycle  is
symplectic, its Lyapunov spectrum is symmetric in the sense that
$\lambda_k=-\lambda_{2g-k+1}$ for all $k=1, \dots, 2g$.

Note that $\mathcal{C}$ is isometrically immersed (usually embedded) into the
corresponding  moduli  space of curves with respect to the hyperbolic
metric  on  the Teichm\"uller curve $\mathcal{C}$ and Teichm\"uller metric in
the moduli space. The cocycle described above is a particular case of
a  more general cocycle related to the Teichm\"uller geodesic flow on
the  moduli  space  (sometimes  called the \textit{Kontsevich--Zorich
cocycle}).

The  Lyapunov exponents of this cocycle are important in the study of
the   dynamics   of  flows  on  surfaces  and  of  interval  exchange
transformations. They were studied by many authors including A. Avila
and  M. Viana \cite{Avila:Viana};  M.  Bainbridge  \cite{Bainbridge};
I. Bouw              and             M. M\"oller \cite{Bouw:Moeller};
G. Forni \cite{Forni1}--\cite{ForniSurvey};  A. Eskin,  M. Kontsevich
and    A. Zorich \cite{Eskin:Kontsevich:Zorich, Kontsevich, Zorich:how:do};            
W. Veech \cite{Veech};           see
surveys \cite{ForniSurvey} and \cite{Zorich:Houches} for an overview.
In  particular,  from  elementary geometric arguments it follows that
one always has $\lambda_1=1$.

We  need  two  results from \cite{Eskin:Kontsevich:Zorich} concerning
the  sum  $\lambda_1+\dots+\lambda_g$  of  all  nonnegative  Lyapunov
exponents  of  the  Hodge  bundle $H^1$ along the geodesic flow on an
arithmetic Teichm\"uller curve.

\begin{theorem}[\cite{Eskin:Kontsevich:Zorich}]
The  sum  of  all  nonnegative Lyapunov exponents of the Hodge bundle
$H^1$ along the geodesic flow on an arithmetic Teichm\"uller curve in
a  stratum  $\mathcal{H}(m_1,\dots,m_n)$,  where  $m_1+\dots+m_n=2g-2$,
satisfies the following relation:
\begin{multline}
\label{eq:general:sum:of:exponents:for:Abelian}
1+\lambda_2 + \dots + \lambda_g
\ = \
\cfrac{1}{12}\cdot\sum_{i=1}^n\cfrac{m_i(m_i+2)}{m_i+1}
\\+\
\cfrac{1}{\operatorname{card}(\textrm{SL}(2,\mathbb{Z})\cdot S_0)}\
\sum_{S_i\in\textrm{SL}(2,\mathbb{Z})\cdot S_0}\ \
\sum_{\substack{
\mathit{horizontal}\\
\mathit{cylinders\ cyl}_{ij}\\
such\ that\\S_i=\sqcup\mathit{cyl}_{ij}}}\
\cfrac{h_{ij}}{w_{ij}}\,.
\end{multline}
where  $S_0$ is a square-tiled surface representing the corresponding
arithmetic Teichm\"uller curve.
\end{theorem}

\begin{remark}
Note  that  the  sum  of  the top $g$ Lyapunov exponents of the Hodge
bundle  coincides  with  a  single  positive Lyapunov exponent of the
complex    line    bundle    $\Lambda^g    H^{1,0}$    often   called
\textit{determinant} line bundle.
\end{remark}

When   an   arithmetic  Teichm\"uller  curve  belongs  to  a  stratum
$\mathcal{Q}(d_1,\dots,d_n)$  of quadratic differentials, one can consider
the same vector bundle $H^1$ as above and define the cocycle and the
Lyapunov  exponents  exactly  in  the  same  way  as  for holomorphic
$1$-forms.  By  reasons which will become clear below, it is
convenient to denote these Lyapunov exponents by
$\lambda^+_1\ge\dots\ge\lambda^+_{2g}$.    Since   the   cocycle   is
symplectic,        we       again       have       the       symmetry
$\lambda^+_k=-\lambda^+_{2g-k+1}$ for all $k=1,\dots,2g$.
However,     for quadratic differentials  one  has $\lambda^+_1<1$ for
any invariant suborbifold (in fact, for any invariant probability measure, see \cite{Forni}).

Following  M. Kontsevich \cite{Kontsevich},  in the case of quadratic
differentials  one  can  define one more vector bundle, $H^1_-$, over
$\mathcal{Q}(d_1,\dots,d_n)$.   Consider  a  pair  (Riemann  surface  $S$,
quadratic  differential  $q$)  representing  a  point of a stratum of
quadratic  differentials  $\mathcal{Q}(d_1,\dots,d_n)$. By assumption, $q$
is  not  a  global  square  of  a  $1$-form.  There  exists a canonical
(possibly   ramified)  double  cover  $p_2:\hat  S\to  S$  such  that
$p_2^\ast  q=\omega^2$,  where  $\omega$ is a holomorphic $1$-form on
$\hat  S$.  Following \cite{Kontsevich}  it  will be convenient to
introduce  the  following  notation. Let $\hat g$ be the genus of the
cover  $\hat  S$.  By  \textit{effective  genus} we call the positive
integer
\begin{equation}
\label{eq:g:eff}
g_{\mathit{eff}}:=\hat g - g\ .
\end{equation}

The  cohomology  space  $H^1(\hat  S,\mathbb{R})$  splits  into  a direct sum
$H^1(\hat S,\mathbb{R})=H_+^1(\hat S,\mathbb{R})\oplus H_-^1(\hat S,\mathbb{R})$ of invariant
and  anti-invariant  subspaces  with  respect  to the action
$H^1(\hat  S,\mathbb{R})\to H^1(\hat S,\mathbb{R})$ induced on cohomology  by  the canonical
involution which commutes  with  the double covering map $p_2: \hat S\to S$.
Note that the invariant part is canonically isomorphic to the cohomology of the
underlying surface, that is, $H_+^1(\hat S,\mathbb{R})\simeq H^1(S,\mathbb{R})$.

We  consider  the subspaces  $H_+^1(\hat S,\mathbb{R})$ and $H_-^1(\hat S,\mathbb{R})$ as
fibers   of   natural   vector   bundles  $H^1_+$  and  $H^1_-$  over
$\mathcal{Q}(d_1,\dots,d_n)$. The bundle $H^1_+$ is canonically isomorphic
to  the  bundle  $H^1$.  The  splitting  $H^1=H^1_+\oplus  H^1_-$  is
equivariant   with   respect  to  the  Gauss--Manin  connection.  The
symplectic  form  restricted  to each summand is nondegenerate. Thus,
the  monodromy  of  the  Gauss--Manin  connection  on  $H_+^1$ and on
$H_-^1$  along  Teichm\"uller  geodesic  flow  defines two symplectic
cocycles.  Following  the  notations established above, we denote the
Lyapunov   exponents   of   the   cocycle   acting   on   $H_+^1$  by
$\lambda^+_1\ge\dots\ge\lambda^+_{2g}$  and  the  ones of the cocycle
acting on $H_-^1$ by $\lambda^-_1\ge\dots\ge\lambda^-_{2g_{\mathit{eff}}}$.
As always for symplectic  cocycles  we  have  the symmetries $\lambda^+_k=-\lambda^+_{2g-k+1}$
and $\lambda^-_k=-\lambda^-_{2g_{\mathit{eff}}-k+1}$.
It  follows  from  the  analogous  result  for  the  case  of Abelian
differentials that
one always has $\lambda^-_1=1$.

Unlike  $H^1_+$,  the vector bundle $H^1_-$ on a stratum of quadratic
differentials  is  not  induced from a vector bundle on an underlying
moduli   space   of   curves.   However,   it  can  be  descended  to
$\mathcal{Q}(d_1,\dots,d_n)/\mathbb{C}^\ast$,  where  $\mathbb{C}^\ast$ is identified with
the   subgroup   of   $\textrm{GL}(2,\mathbb{R})$   acting  on  $\mathcal{Q}(d_1,\dots,d_n)$  by
multiplying  a  quadratic  differential  by  a  nonzero  constant. In
particular,  for any Veech surface $(S,q)$ in $\mathcal{Q}(d_1,\dots,d_n)$
the  bundle  $H^1_-$  can be descended from the $\textrm{PSL}(2,\mathbb{R})$-orbit $\mathcal{O}$ of
$(S,q)$ in $\mathcal{Q}(d_1,\dots,d_n)$ to the corresponding Teichm\"uller
curve $\mathcal{C}=\mathcal{O}/\mathbb{C}^\ast$.

\begin{theorem}[\cite{Eskin:Kontsevich:Zorich}]
The sums of all nonnegative Lyapunov exponents of the bundles $H_+^1$
and  $H^1_-$  along  the geodesic flow on an arithmetic Teichm\"uller
curve  in a stratum $\mathcal{Q}(d_1,\dots,d_n)$, of meromorphic quadratic
differentials  with  at  most  simple  poles,  satisfy  the following
relations:
\begin{multline}
\label{eq:general:plus:sum:of:exponents:for:quadratic}
\lambda^+_1+\lambda^+_2 + \dots + \lambda^+_g
\ = \
\cfrac{1}{24}\cdot\sum_{i=1}^n\cfrac{d_i(d_i+4)}{d_i+2}
\\+\
\cfrac{1}{\textrm{card}(\textrm{PSL}(2,\mathbb{Z})\cdot S_0)}\
\sum_{S_i\in\textrm{PSL}(2,\mathbb{Z})\cdot S_0}\ \
\sum_{\substack{
\mathit{horizontal}\\
\mathit{cylinders\ cyl}_{ij}\\
such\ that\\S_i=\sqcup\mathit{cyl}_{ij}}}\
\cfrac{h_{ij}}{w_{ij}}\,.
\end{multline}
for the Lyapunov exponents of the bundle $H^1=H^1_+$ and

\begin{multline}
\label{eq:general:minus:sum:of:exponents:for:quadratic}
1+\lambda^-_2 + \dots + \lambda^-_{g_{\mathit{eff}}}
\ = \
\cfrac{1}{24}\cdot\sum_{i=1}^n\cfrac{d_i(d_i+4)}{d_i+2}
\ +\
\cfrac{1}{4}\,\cdot\,\sum_{\substack{j \text{ such that}\\
d_j \text{ is odd}}}
\cfrac{1}{d_j+2}
\\+\
\cfrac{1}{\textrm{card}(\textrm{PSL}(2,\mathbb{Z})\cdot S_0)}\
\sum_{S_i\in\textrm{PSL}(2,\mathbb{Z})\cdot S_0}\ \
\sum_{\substack{
\mathit{horizontal}\\
\mathit{cylinders\ cyl}_{ij}\\
such\ that\\S_i=\sqcup\mathit{cyl}_{ij}}}\
\cfrac{h_{ij}}{w_{ij}}\,.
\end{multline}
for  the  Lyapunov  exponents  of the bundle $H^1_-$. Here $S_0$ is a
square-tiled   surface   representing  the  corresponding  arithmetic
Teichm\"uller curve, and $g$ and $g_{\mathit{eff}}$ are the genus and
the effective genus \eqref{eq:g:eff} of $S_0$.
\end{theorem}

\subsection{Sum  of  the Lyapunov exponents for a square-tiled cyclic
cover}
\label{ss:Sum:of:exponents}

Now  everything is ready to apply the results of the previous section
to square-tiled cyclic covers.

\begin{theorem}
\label{th:sum:of:exponents:for:cyclic:Abelian}
Let us consider  an  even  integer  $N$  and  a  collection  of odd integers
$a_1,a_2,a_3,a_4$ satisfying the relations in formula \eqref{eq:a1:a4}.

The  sum  of  all  nonnegative Lyapunov exponents of the Hodge bundle
$H^1$  along  the geodesic flow on the arithmetic Teichm\"uller curve
of the square-tiled cyclic cover $M_N(a_1,a_2,a_3,a_4)$ is expressed by the formula below:
\begin{multline}
\label{eq:sum:of:exponents:for:cyclic:Abelian}
1+\lambda_2 + \dots + \lambda_g
\ = \
\frac{N}{6}-
\frac{1}{6 N}\,\sum_{i=1}^4 {\gcd}^2(N,a_i)
\\+\
\cfrac{1}{6N}
\bigg({\gcd}^2(N,a_1+a_2)+{\gcd}^2(N,a_1+a_3)+{\gcd}^2(N,a_1+a_4)\bigg)
\end{multline}
\end{theorem}
\begin{proof}
We apply  formula \eqref{eq:general:sum:of:exponents:for:Abelian}
taking into account the following data. The singularity pattern
$(m_1,\dots,m_n)$  of  the  holomorphic $1$-form corresponding to
the square-tiled cyclic cover
$M_N(a_1,a_2,a_3,a_4)$  is computed in Lemma
\ref{lm:singularity:pattern}, see formula
\eqref{eq:singularities:one:form}.   The
$\textrm{SL}(2,\mathbb{Z})$-orbit of the square-tiled surface
$M_N(a_1,a_2,a_3,a_4)$ is described by Theorem \ref{thm:1prime} and
the cylinder  decomposition for each square-tiled surface in the
orbit is given in Lemma \ref{lm:cylinder:width}.    By plugging the
above data in formula \eqref{eq:general:sum:of:exponents:for:Abelian}
we obtain formula \eqref{eq:sum:of:exponents:for:cyclic:Abelian}.
\end{proof}

\begin{theorem}
\label{th:sum:of:exponents:for:cyclic:quadratic}
Let us consider integers $N$ and $a_1,\dots,a_4$ satisfying
the relations in formula \eqref{eq:a1:a4}.  Let us suppose, in addition, that either $N$ is
odd,  or  $N$ is even and at least one of $a_i$, $i=1,2,3,4$, is also
even.

The  sum  of all nonnegative Lyapunov exponents of the bundle $H_+^1$
along  the geodesic flow on the arithmetic Teichm\"uller curve of the
square-tiled cyclic cover $M_N(a_1,a_2,a_3,a_4)$ is expressed by the following formula:
\begin{multline}
\label{eq:sum:of:plus:exponents:for:cyclic:quadratic}
\lambda^+_1+\lambda^+_2 + \dots + \lambda^+_g
\ = \
\frac{N}{6}-
\frac{1}{6 N}\,\sum_{i=1}^4 {\gcd}^2(N,a_i)
\\+\
\cfrac{1}{6N}
\big({\gcd}^2(N,a_1+a_2)+{\gcd}^2(N,a_1+a_3)+{\gcd}^2(N,a_1+a_4)\big)
\end{multline}

The  sum  of all nonnegative Lyapunov exponents of the bundle $H_-^1$
along  the geodesic flow on the arithmetic Teichm\"uller curve of the
square-tiled cyclic cover $M_N(a_1,a_2,a_3,a_4)$ is expressed by the following formula:
\begin{multline}
\label{eq:sum:of:minus:exponents:for:cyclic:quadratic}
1+\lambda^-_2 + \dots + \lambda^-_{g_{\mathit{eff}}}
\ = \frac{N}{6}  \\
\ +\
\frac{1}{12 N}\,
\sum_{\substack{
i\text{ such that }\\
\vspace*{-.5\baselineskip}\\
\frac{N}{\gcd(N,a_i)}\text{ is odd}
}}
{\gcd}^2(N,a_i)
\ -\
\frac{1}{6 N}\,
\sum_{\substack{
i\text{ such that }\\
\vspace*{-.5\baselineskip}\\
\frac{N}{\gcd(N,a_i)}\text{ is even}
}}
{\gcd}^2(N,a_i)
\\+\
\cfrac{1}{6N}
\big({\gcd}^2(N,a_1+a_2)+{\gcd}^2(N,a_1+a_3)+{\gcd}^2(N,a_1+a_4)\big)
\end{multline}
\end{theorem}
\begin{proof}
We apply formulae \eqref{eq:general:plus:sum:of:exponents:for:quadratic} and \eqref{eq:general:minus:sum:of:exponents:for:quadratic} taking into account the following data.
The singularity pattern  $(d_1,\dots,d_n)$ of the quadratic differential corresponding to
the square-tiled cyclic cover $M_N(a_1,a_2,a_3,a_4)$ is computed in Lemma \ref{lm:singularity:pattern},
see formula \eqref{eq:singularities:quadratic:differential}.
The $\textrm{PSL}(2,\mathbb{Z})$-orbit  of the square-tiled cyclic surface $M_N(a_1,a_2,a_3,a_4)$  is  described  by  Theorem \ref{th:SLZ:orbit}  and a cylinder decomposition for each square-tiled surface in the orbit is given
in Lemma \ref{lm:cylinder:width}. By plugging the above data in \eqref{eq:general:plus:sum:of:exponents:for:quadratic} and \eqref{eq:general:minus:sum:of:exponents:for:quadratic} we
obtain \eqref{eq:sum:of:plus:exponents:for:cyclic:quadratic} and  \eqref{eq:sum:of:minus:exponents:for:cyclic:quadratic} respectively.
\end{proof}

\begin{remark}
\hspace*{-1.47pt}
Actually, the Hodge bundles $H^{1,0}$ and $H^1_{\mathbb{C}}$ over the
Teichm\"uller curve of a square-tiled cyclic cover have a very
explicit decomposition into a direct sum of one- and two-dimensional
vector subbundles, see \cite{Bouw}. A similar decomposition, used
also in \cite{Bouw:Moeller}, \cite{McMullen} and \cite{CMFZ-survey},
enables, in particular, to compute explicitly all \textit{individual}
Lyapunov exponents for any square-tiled cyclic cover,
see \cite{Eskin:Kontsevich:Zorich:cyclic}.
\end{remark}

\subsection{Degenerate Lyapunov spectrum}
\label{ss:degenerate:spectrum}
In this section we list of all examples of arithmetic Teichm\"uller
curves coming from square-tiled cyclic covers with  \emph{maximally
degenerate} Lyapunov spectra. We recall  that  by  elementary
geometric reasons in strata of Abelian differentials $\lambda_1$ is
equal to one, while in strata of quadratic differentials (which are
not squares) $\lambda_1^-$  is  equal  to  one,  for any ergodic
invariant measure. Thus, for strata of Abelian differentials we speak
of     ``maximally degenerate spectrum'' whenever
$\lambda_2=\dots=\lambda_g=0$, while for strata of quadratic
differentials (which are not squares) we speak of     ``maximally
degenerate spectrum'' of $\lambda^-$-exponents,   whenever
$\lambda^-_2=\dots=\lambda^-_{g_{\mathit{eff}}}=0$ and ``maximally
degenerate  spectrum'' of $\lambda^+$-exponents, whenever
$\lambda^+_1=\dots=\lambda^+_g=0$.

\subsubsection{Abelian Differentials}

We start with $M_N(a_1,a_2,a_3,a_4)$ square-tiled cyclic covers which
give rise to holomorphic $1$-forms. By Lemma \ref{lm:singularity:pattern}  this
corresponds to even $N$ and odd $a_i$, $i=1,2,3,4$.

\begin{theorem}
\label{th:degenerate:spectrum:Abelian}
The  cyclic  cover $M_2(1,1,1,1)$ has genus one, so there is a single
nonnegative  Lyapunov  exponent $\lambda_1$ of the Hodge bundle $H^1$
along the geodesic flow on the arithmetic Teichm\"uller curve of this
cyclic cover; as always $\lambda_1=1$.

For   the   arithmetic  Teichm\"uller  curves  corresponding  to  the
square-tiled cyclic covers
$$
M_4(1,1,1,1)\simeq M_4(3,3,3,3)
\quad\text{ and }\quad
M_6(1,1,1,3)\simeq M_6(5,5,5,3)
$$
the Lyapunov spectrum is maximally degenerate,
that is $\lambda_2=\dots=\lambda_g=0$.

For  all other cyclic covers of the form $M_N(a_1,a_2,a_3,a_4)$ with even $N$ and odd
$a_i$, $i=1,2,3,4$, one has $\lambda_2>0$.
\end{theorem}

\begin{remark}
The  fact  that  the Lyapunov spectrum of $M_4(1,1,1,1)$ is maximally
degenerate  was  discovered  by  G. Forni  in \cite{ForniSurvey} by a
symmetry  argument.  Later  G. Forni and C. Matheus discovered by the
same  approach  that  the Lyapunov spectrum of $M_6(1,1,1,3)$ is also
maximally       degenerate,       see \cite{Forni:Matheus}       (and
also \cite{CMFZ-survey}).
\end{remark}

\begin{remark}
M. M\"oller \cite{Moeller}  has  an  independent  and by far stronger
result  showing  that  the two above-mentioned examples of arithmetic
Teichm\"uller  curves  with  a maximally degenerate Lyapunov spectrum
are  really  exceptional;  see Conjecture \ref{conj:the:only:two} and
Remarks \ref{rm:Moeller}  and \ref{rm:EKZ:bound}  below.
\end{remark}

\begin{proof}
Applying   formula \eqref{eq:sum:of:exponents:for:cyclic:Abelian}  to
$M_4(1,1,1,1)$    and    $M_6(1,1,1,3)$    we    get    a    relation
$1+\lambda_2+\dots+\lambda_g=1$.                                Since
$\lambda_2\ge\dots\ge\lambda_g\ge  0$,  this  implies that, actually,
$\lambda_2=\dots=\lambda_g= 0$.

It  remains to prove that for all other collections $N,a_1,\dots,a_4$
the                 right-hand                 side                of
formula \eqref{eq:sum:of:exponents:for:cyclic:Abelian}   is  strictly
greater                than               $1$.               Applying
formula \eqref{eq:sum:of:exponents:for:cyclic:Abelian}  we  see  that
this  statement is valid for the remaining two collections for $N=4$.
Now we can assume that $N\ge 6$.

Since  $\gcd(N,a_i)$  is  a  divisor  of  $N$,  and  $1\le a_i<N$, we
conclude  that $\gcd(N,a_i)$ might be $N/2$, $N/3$ or less. Hence, we
always have
$$
\sum_{i=1}^4 {\gcd}^2(N,a_i)\le N^2\ .
$$
This  implies  that  if we have $\gcd(N,a_i)=\gcd(N,a_j)=N/2$ for  two  distinct indices $i\neq j$, then $a_i=a_j=N/2$ and at least one of
the summands in
$$
{\gcd}^2(N,a_1+a_2)+{\gcd}^2(N,a_1+a_3)+{\gcd}^2(N,a_1+a_4)
$$
is  equal  to  $N^2$.  This  means  that  the expression on the right
of \eqref{eq:sum:of:exponents:for:cyclic:Abelian} is strictly greater
than $1$.

Thus,  we  can  assume  that  there  is  at  most one $a_i$ such that
$\gcd(N,a_i)=N/2$,  while  for  the  other  indices $j\neq i$ we have
$\gcd(N,a_j)\le N/3$. Then
$$
\frac{N}{6}-\frac{1}{6N}\,\sum_{i=1}^4 {\gcd}^2(N,a_i)\ge
\frac{5N}{72}\ .
$$
For  $N\ge  16$  we  have  $5N/72>1$,  and  hence,  for $N\ge 16$ the
expression                 on                the                right
of \eqref{eq:sum:of:exponents:for:cyclic:Abelian} is strictly greater
than $1$.

It    remains    to    consider   finite   number   of   arrangements
$N,a_1,\dots,a_4$,  with $6\le N\le 14$. This can be done either by a
straightforward  check,  or  by considerations similar to the ones as
above.
\end{proof}

\begin{figure}[htb]
   %
   %
\includegraphics{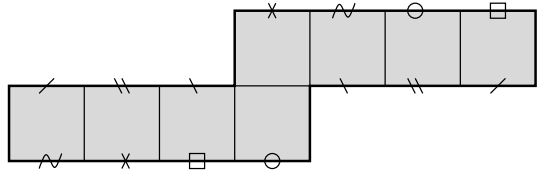}
\vspace{60pt}
\caption{
\label{fig:Eierlegende:Wollmilchsau}
Eierlegende Wollmilchsau}
\end{figure}

\begin{remark}
Figure \ref{fig:Eierlegende:Wollmilchsau}  presents  the square-tiled
cyclic  cover  $M_4(1,\!1,\!1,\!1)$.  This surface is also a $2$-fold
cover  over a torus branched at four points. It has plenty of unusual
properties.  In  particular,  it  was  noticed by T. Monteil that for
every  saddle  connection  there  is  always a twin saddle connection
with  the  same  length  and  the  same direction, and there are no
simple   saddle   connections   joining   a  singularity  to  itself.
M. M\"oller   has   proved  that  the  Teichm\"uller  curves  of  the
square-tiled cyclic covers $M_4(1,1,1,1)$,
presented in Figure \ref{fig:Eierlegende:Wollmilchsau},
and $M_6(1,1,1,3)$,
presented in Figure \ref{fig:cartoon},
are also Shimura curves \cite{Moeller}.
\end{remark}

\begin{conjecture}
\label{conj:the:only:two}
The  only Teichm\"uller curves in the strata of Abelian differentials
in genus $g\ge 2$ with maximally degenerate Lyapunov spectra, that
is,  such  that  $\lambda_2=\dots=\lambda_g=0$,  are  the  arithmetic
Teichm\"uller  curves  of  the  cyclic  coverings  $M_4(1,1,1,1)$ and
$M_6(1,1,1,3)$.
\end{conjecture}

\begin{remark}
\label{rm:Moeller}
According  to M. M\"oller \cite{Moeller}, the conjecture holds in all
genera different from five and for some strata in genus five. For the
remaining  strata  in  genus  five  the statement requires some extra
verification.
\end{remark}

\begin{problem}\hspace*{-3.5pt}
Are  there  any  closed  invariant suborbifolds (of any dimension) in
strata  of  Abelian  differentials with maximally degenerate Lyapunov
spectra  ($\lambda_2=\dots=\lambda_g=0$)  different  from  the  two
Teichm\"uller curves presented above?
\end{problem}

\begin{remark}
\label{rm:EKZ:bound}
It is proved in \cite{Eskin:Kontsevich:Zorich} that there are no such
regular  $\textrm{SL}(2,\mathbb{R})$-invariant  suborbifolds  in  \textit{any}  stratum  of
Abelian  differentials  of  genus $7$ and higher and in \textit{some}
strata in genera $5$ and $6$.
\end{remark}

\begin{remark}
In  genus  $g=2$, M. Bainbridge \cite{Bainbridge} has proved that the
second   exponent   $\lambda_2$  equals  to  $1/2$  for  all  ergodic
$\textrm{SL}(2,\mathbb{R})$-invariant   measures   supported  in  the  stratum  $\mathcal{H}(1,1)$,
corresponding  to  two  simple zeros of the holomorphic differential,
and  $\lambda_2$  equals  to  $1/3$  for  all ergodic $\textrm{SL}(2,\mathbb{R})$-invariant
measures  supported  in  the  stratum  $\mathcal{H}  (2)$, corresponding to a
double  zero.  In  particular,  in genus $2$ the Lyapunov spectrum is
always  non-degenerate  and  simple.  Bainbridge's result was already
known  conjecturally  since \cite{Kontsevich}  as  a consequence of a
formula  for  the sum of exponents for an $\textrm{SL}(2,\mathbb{R})$-invariant submanifold
of  the  hyperelliptic  locus  in any stratum. Such a formula has now
been proved in \cite{Eskin:Kontsevich:Zorich}.
\end{remark}

Thus,  any  regular  $\textrm{SL}(2,\mathbb{R})$-invariant  suborbifold with maximally
degenerate  Lyapunov spectrum might live in genera $3$ and $4$ and in
some strata in genera $5$ and $6$ only.

\subsubsection{Holomorphic quadratic differentials}
\label{ss:Holomorphic:quadratic:differentials}

Let  us  consider  next square-tiled cyclic co\-vers $M_N(a_1,a_2,a_3,a_4)$, which give
rise to holomorphic quadratic differentials. In particular, we assume
through   Section \ref{ss:Holomorphic:quadratic:differentials}   that
inequalities in formula (\ref{eq:a1:a4}) are strict.

\begin{theorem}
\label{th:degenerate:spectrum:quadratic}
A  square-tiled cyclic cover $M_4(3,2,2,1)$ in the stratum $\mathcal{Q}(2,2)$
has  effective  genus  one, so there is a single nonnegative Lyapunov
exponent  $\lambda^-_1$  of  the  vector  bundle  $H^1_-$  along  the
geodesic  flow  on  the arithmetic Teichm\"uller curve of this cyclic
cover; as always $\lambda^-_1=1$.

The Lyapunov spectrum of the vector bundle $H^1_-$ along the geodesic
flow  on  the arithmetic Teichm\"uller curves of the following cyclic
covers
\begin{itemize}
\item
$M_5(2,1,1,1)\simeq M_5(4,2,2,2)\simeq M_5(1,3,3,3)\simeq M_5(3,4,4,4)$
in
$\mathcal{Q}(3,3,3,3)$;
\item
$M_6(5,3,2,2)\simeq M_6(1,3,4,4)$ in the stratum $\mathcal{Q}(4,1,1,1,1)$;
\item
$M_8(4,2,1,1)\simeq M_8(4,6,3,3)\simeq M_8(4,2,5,5)\simeq M_8(4,6,7,7)$
and
$M_8(7,4,3,2)\simeq M_8(1,4,5,6)$
in the stratum $\mathcal{Q}(6,6,2,2)$
\end{itemize}
is maximally degenerate, that is, $\lambda^-_2=\dots=\lambda^-_{g_{\mathit{eff}}}=0$.

For  all other cyclic covers of the form $M_N(a_1,a_2,a_3,a_4)$ with odd $N$, or with
even   $N$  and  at  least  one  even  $a_i$,  $i=1,2,3,4$,  one  has
$g_{\mathit{eff}}\ge  2$  and  $\lambda^-_2>0$.  Here  we assume that
$0<a_i<N$ for all $i=1,2,3,4$.
\end{theorem}
\begin{proof}
Applying
formula \eqref{eq:sum:of:minus:exponents:for:cyclic:quadratic} to the
cyclic  covers  from  the  list  given  in  the Theorem we check that
$1+\lambda^-_2+\dots+\lambda^-_{g_{\mathit{eff}}}=1$.           Since
$\lambda^-_2\ge\dots\ge\lambda^-_{g_{\mathit{eff}}}\ge 0$ this proves
the  equalities $\lambda^-_2=\dots=\lambda^-_{g_{\mathit{eff}}}=0$ in
the cases mentioned above.
The remaining part of the proof is completely analogous to the one of
Theorem \ref{th:degenerate:spectrum:Abelian}.
\end{proof}

\begin{problem}
\label{pr:holomorphic:quadratic}
Are  there any other Teichm\"uller curves in a stratum of holomorphic
quadratic  differentials  with  effective genus at least two and with
maximally degenerate Lyapunov spectrum of the bundle $H^1_-$?

\hspace*{-3pt}
Are   there  any  closed  invariant  submanifolds  (suborbifolds)  of
dimension greater than one satisfying this property?
\end{problem}

\begin{remark}
The  only  holomorphic  quadratic  differentials in genus one are the
squares  of holomorphic $1$-forms. In genus two the strata $\mathcal{Q}(4)$ and
$\mathcal{Q}(3,1)$   are   empty,   see \cite{Masur:Smillie}.   The   stratum
$\mathcal{Q}(2,2)$  in  genus  two has effective genus one. The remaining two
strata,  namely $\mathcal{Q}(2,1,1)$ and $\mathcal{Q}(1,1,1,1)$, have effective genera
two  and  three  respectively;  both  of  them  are hyperelliptic,
see \cite{Lanneau}.  It  is  proved in \cite{Eskin:Kontsevich:Zorich}
that  one  has  $\lambda^-_2  =1/3$  for any regular $\textrm{PSL}(2,\mathbb{R})$-invariant
suborbifold       in       $\mathcal{Q}_1(2,1,1)$       and      one      has
$\lambda_2^-+\lambda_3^-=2/3$   for   any   regular  $\textrm{PSL}(2,\mathbb{R})$-invariant
suborbifold in $\mathcal{Q}_1(1,1,1,1)$.

It  is  proved  in \cite{Eskin:Kontsevich:Zorich}  that  there are no
regular $\textrm{PSL}(2,\mathbb{R})$-invariant suborbifolds
with maximally  degenerate Lyapunov spectra on the bundle $H^1_-$
in  \textit{any}  stratum  of  holomorphic quadratic differentials of
genus  $7$  and  higher and in \textit{some} strata in genera $5$ and
$6$. Thus, if the answer to Problem \ref{pr:holomorphic:quadratic} is
affirmative,   the   corresponding   invariant   suborbifold   should
correspond  to  genera  $3$  or  $4$  or to some particular strata in
genera $5$ or $6$.
\end{remark}

A formula in \cite{Eskin:Kontsevich:Zorich}, which generalizes
formula \eqref{eq:general:plus:sum:of:exponents:for:quadratic},
shows that one always has  $\lambda_1^+>0$, hence the spectrum of
Lyapunov exponents of the subbundle $H^1_+$ over an  $\textrm{PSL}(2,\mathbb{R})$-invariant
submanifold in a stratum of \textit{holomorphic} quadratic
differentials  can never be maximally degenerate. Theorem \ref{th:degenerate:spectrum:quadratic:meromorphic}   in   the
section  below  shows  that  the  situation  is  different for strata
of \textit{meromorphic} quadratic differentials with at most simple
poles.

\subsubsection{Meromorphic quadratic differentials}

Let us consider  square-tiled  cyclic  covers  $M_N(a_1,a_2,a_3,a_4)$  which  give  rise to
meromorphic  quadratic  differentials  with  simple  poles. This case
corresponds  to collections $(N;a_1,a_2,a_3,a_4)$ of parameters where
at least one of the $a_i$'s is equal to $N$.

\begin{theorem}
\label{th:degenerate:spectrum:quadratic:meromorphic}
If  $a_i=N$  for  at  least  one  of  $i=1,2,3,4$,  then the Lyapunov
spectrum  of the vector bundle $H^1_+$ along the geodesic flow on the
arithmetic     Teichm\"uller    curve    of    the    cyclic    cover
$M_N(a_1,a_2,a_3,a_4)$   is   maximally   degenerate:   all  Lyapunov
exponents $\lambda^+_i$ are equal to zero.
\end{theorem}
\begin{proof}
Without  loss  of  generality  we  may  assume that $a_1=N$. Applying
formula \eqref{eq:sum:of:plus:exponents:for:cyclic:quadratic} we get:
\begin{multline*}
\lambda^+_1+\lambda^+_2 + \dots + \lambda^+_g
\ = \
\frac{N}{6}-\frac{{\gcd}^2(N,N)}{6 N}-
\frac{1}{6 N}\,\sum_{i=2}^4 {\gcd}^2(N,a_i)
\\+\
\cfrac{1}{6N}
\bigg({\gcd}^2(N,N+a_2)+{\gcd}^2(N,N+a_3)+{\gcd}^2(N,N+a_4)\bigg)
\\=\
-\frac{1}{6 N}\,\sum_{i=2}^4 {\gcd}^2(N,a_i)
+\frac{1}{6 N}\,\sum_{i=2}^4 {\gcd}^2(N,a_i)
\ =\ 0
\end{multline*}

\end{proof}

\begin{theorem}
\label{th:degenerate:spectrum:quadratic:meromorphic:special}
The  square-tiled  cyclic cover $M_2(2,2,1,1)$ belongs to the stratum
$\mathcal{Q}(-1^4)$.   There   is  a  single  nonnegative  Lyapunov  exponent
$\lambda^-_1=1$  of the vector bundle $H^1_-$ along the geodesic flow
on  the  arithmetic  Teichm\"uller  curve of this cyclic cover and no
other Lyapunov exponents.

The  Lyapunov spectrum of the vector bundle $H^1_-\oplus H^1_+$ along
the  geodesic  flow  on  the  arithmetic  Teichm\"uller  curve of the
following cyclic covers in genus one
\begin{itemize}
\item
$M_3(3,1,1,1)\simeq M_3(3,2,2,2)$
in the stratum $\mathcal{Q}(1^3,-1^3)$;
\item
$M_4(4,2,1,1)\simeq M_4(4,2,3,3)$ in the stratum $\mathcal{Q}(2^2,-1^4)$
\end{itemize}
 is maximally degenerate, that is,
$\lambda^-_2=\dots=\lambda^-_{g_{\mathit{eff}}}=0$ and
$\lambda^+_1=0$.

For  all other cyclic covers of the form $M_N(N,a_2,a_3,a_4)$ one has
$g_{\mathit{eff}}\ge 2$ and $\lambda^-_2>0$.
\end{theorem}
\begin{proof}
Without   loss  of  generality  we  may  assume  that  $a_1=N\ge  3$.
Formula \eqref{eq:sum:of:minus:exponents:for:cyclic:quadratic}
applied to this case can be rewritten as follows:
\begin{multline*}
1+\lambda^-_2 + \dots + \lambda^-_{g_{\mathit{eff}}}
\ = \ \frac{N}{4}  \\  \ +\
\frac{1}{12 N}\,
\sum_{\substack{
i\ge 2\text{ such that }\\
\vspace*{-.5\baselineskip}\\
\frac{N}{\gcd(N,a_i)}\text{ is odd}
}}
{\gcd}^2(N,a_i)
\ -\
\frac{1}{6 N}\,
\sum_{\substack{
i\text{ such that }\\
\vspace*{-.5\baselineskip}\\
\frac{N}{\gcd(N,a_i)}\text{ is even}
}}
{\gcd}^2(N,a_i)
 \\ \ +\
\cfrac{1}{6N}
\bigg({\gcd}^2(N,a_2)+{\gcd}^2(N,a_3)+{\gcd}^2(N,a_4)\bigg)
\ \ge\
\frac{N}{4}
\end{multline*}
Hence,  for  $N>4$ we get $\lambda^-_2>0$. Applying the above formula
to remaining data for $N=2,3,4$ we complete the proof of the Theorem.
\end{proof}

\begin{remark}
The          maximally         degenerate         examples         in
Theorem \ref{th:degenerate:spectrum:quadratic:meromorphic:special}
have  the  following  origin.  For each of the cyclic covers as above
consider  a  canonical  ramified  double cover, such that the induced
quadratic  differential  is  a global square of a holomorphic $1$-form.
The  resulting square-tiled surface is isomorphic to the square-tiled
cyclic cover $M_6(3,1,1,1)$ for the double cover over $M_3(3,1,1,1)$,
and  also to  the  square-tiled cyclic cover $M_4(1,\!1,\!1,\!1)$
for the double
cover     over     $M_4(4,2,1,\!1)$.     Thus    the    statement    of
Theorem \ref{th:degenerate:spectrum:quadratic:meromorphic:special}
for       such       examples      follows      immediately      from
Theorem \ref{th:degenerate:spectrum:Abelian}.
\end{remark}

\begin{problem}
Are  there  other  Teichm\"uller  curves  in  strata  of  meromorphic
quadratic  differentials  with  at  most  simple poles with effective
genus at least two and with maximally degenerate Lyapunov spectrum of
the  bundle  $H^1_+$?  Same question for the bundle $H^1_-$? For both
bundles $H^1_-$ and $H^1_+$ simultaneously?

\hspace*{-2.5pt}
Are   there  any  closed  invariant  submanifolds  (suborbifolds)  of
dimension greater than one satisfying this property?
\end{problem}

\appendix
\section{Parity of the spin structure}
\label{a:parity:of:spin}

The  ambient  stratum $\mathcal{H}(1,1,1,1)$ for the Eierlegende Wollmilchsau
representing the square-tiled $M_4(1,1,1,1)$ is connected.
The   ambient  stratum  $\mathcal{H}(2,2,2)$  for  the  square-tiled  surface
corresponding   to  the  cyclic  cover  $M_6(1,1,1,3)$  contains  two
connected   components  representing  even  and  odd  parity  of  the
spin-structure    of    the    corresponding    holomorphic   $1$-form,
see \cite{Kontsevich:Zorich}.
The  corresponding  parity  of the spin-structure of the square-tiled
cyclic  cover $M_6(1,1,1,3)$ was computed in \cite{Matheus:Yoccoz} by
combinatorial  methods.  We present another calculation to illustrate
an  alternative  analytic technique, which is less known in the dynamical
community.

\begin{proposition}
A   holomorphic  $1$-form  $\omega$  defining  a  square-tiled  surface
associated  to the cyclic cover $M_6(1,1,1,3)$ has even parity of the
spin structure, hence
$$
(M_6(1,1,1,3),\omega)\in \mathcal{H}^{even}(2,2,2)\ .
$$
\end{proposition}
\begin{proof}
Let us consider  a  holomorphic  $1$-form $\omega$ with zeroes of even degrees
only, that is, with a pattern of zeroes of the form $(2d_1, \dots, 2d_n)$ and let
$K(\omega)=2d_1  P_1+\dots+2d_n  P_n$ be its zero divisor. An equivalent
definition of the \textit{parity of the spin structure} $\phi(\omega)$ associated  to
$\omega$ is the dimension of the space of holomorphic $1$-forms with zeroes
of  degrees  at  least $d_1,\dots, d_n$ at $P_1,\dots,P_n$ respectively, computed
modulo $2$, that is,
$$
\phi(\omega):=\dim\left|\frac{1}{2}K(\omega)\right|+1\pmod 2
$$
(see \cite{Atiyah, Johnson, Milnor, Mumford} for
more  information  on  the  spin-structure). Let us compute the above
dimension   for   the   holomorphic   $1$-form  $\omega$  defining  the
square-tiled flat structure $\omega^2=p^\ast q_0$ on $M_6(1,1,1,3)$.

We recall that $M_6(1,1,1,3)$ is defined by the equation
$$
w^6=(z-z_1)(z-z_2)(z-z_3)(z-z_4)^3\ .
$$
Let us consider the following $1$-forms on $M_6(1,1,1,3)$,
\begin{align}
\notag
\alpha(c_1,c_2)&=(c_1z+c_2)(z-z_4)^2\frac{dz}{w^5}\qquad c_1, c_2=const\\
\label{eq:basis:of:hol:1:forms}
\beta&=(z-z_4)\frac{dz}{w^4}\\
\notag
\gamma&=(z-z_4)\frac{dz}{w^3}
\end{align}
We  claim  that  all  these  forms  are  holomorphic  (in the case of
$\alpha(c_1,c_2)$,  it  is  holomorphic  for all values of parameters
$c_1,c_2\in\mathbb{C}$). For example, let  us check this for $\gamma$. When
$w\neq   0$   and   $z\neq\infty$,   the  form  $\gamma$  is  clearly
holomorphic.  In  a  neighborhood  of  any of the ramification points
$P_i$,  where  $i=1,2,3$,  we have $w^6\sim(z-z_i)$, so $dz\sim w^5\,
dw$,  and with respect to a local coordinate $w$ in a neighborhood of
$P_4$,  we  get  $\gamma\sim w^2\,dw$. This shows that $\gamma$ has a
double zero at each of the points $P_1,P_2,P_3$. In a neighborhood of
$z_4$ we have $w^2\sim(z-z_4)$, so with respect to a local coordinate
$w$      we     get     $dz=w\,dw$,     and     hence     $\gamma\sim
w^2\cfrac{w\,dw}{w^3}=dw$.  Hence,  each  of  the  three preimages of
$z_4$ on our Riemann surface is a regular point of $\gamma$. Finally,
choosing  a  local coordinate $t=1/z$ in a neighborhood of $z=\infty$
we  see  that  $(z-z_i)\sim  t^{-1}$, $dz\sim t^{-2}\,dt$, and $w\sim
t^{-1}$,  so  $\gamma\sim dt$. Hence, any preimage of $z=\infty$ is a
regular  point  of  $\gamma$.  In  a  similar  way  we can check that
$\alpha(c_1,c_2)$   and  $\beta$  are  holomorphic  (see \cite{Bouw},
\cite{Bouw:Moeller},  \cite{Eskin:Kontsevich:Zorich:cyclic}  for more
details). Clearly,
\begin{equation}
\label{eq:linear:combination}
\alpha(c_1,c_2)+c_3\beta+c_4\gamma
\end{equation}
is   identically   zero   if   and   only   if   $c_1=c_2=c_3=c_4=0$.
By formula \eqref{eq:genus},  we  know  that  the  genus of $M_6(1,1,1,3)$ is
equal  to  $4$.  Hence,  we  have constructed a basis of the space of
holomorphic  $1$-forms  on $M_6(1,1,1,3)$: every holomorphic $1$-form can
be represented as a linear combination \eqref{eq:linear:combination},
and this representation is unique.

Note  that the forms $\alpha,\beta,\gamma$ are eigenforms of the deck
transformation  $(z,w)\mapsto  (z,\zeta w)$, where $\zeta$ is a sixth
primitive  root  of unity, $\zeta^6=1$. The corresponding eigenvalues
are $\zeta, \zeta^2, \zeta^3$. In particular, since $\zeta^3=-1$, the
form  $\gamma$  is  anti-invariant with respect to a generator of the
group  of  deck  transformations.  Our  calculation  shows  that  the
eigenspace  corresponding  to the eigenvalue $-1$ is one-dimensional.
Hence, the form $\gamma$ differs from the holomorphic $1$-form $\omega$ defining
our   square-tiled   flat  structure  on  $M_6(1,1,1,3)$  only  by  a
(non-zero) multiplicative constant.

Assume that  $c_1,c_2$  are  not simultaneously equal to zero. When
the ratio $-c_2/c_1$ satisfies
$-c_2/c_1\notin\{z_1,z_2,z_3,z_4,\infty\}$,  the  holomorphic  $1$-form
$\alpha(c_1,c_2)$  has  six  simple  zeroes:  one  at each of the six
preimages  of  the root of the polynomial $(c_1z+c_2)$. When $c_1=0$,
the  holomorphic $1$-form $\alpha(c_1,c_2)$ also has six simple zeroes:
one  at each of the six preimages of $z=\infty$. When $-c_2/c_1=z_i$,
where  $i=1,2,3$,  the  form  $\alpha(c_1,c_2)$  has a single zero of
degree   $6$   at  $P_i$.  Finally,  when  $-c_2/c_1=z_4$,  the  form
$\alpha(c_1,c_2)$  has three zeroes of degree two: one at each of the
three preimages of $z=z_4$.

The  holomorphic  form  $\beta$ has $6$ simple zeroes: one at each of
the six ramification points $z=z_i$, where $i=1,2,3,4$.

Our       consideration       implies       that       a       linear
combination \eqref{eq:linear:combination} has a zero at \textit{each}
ramification  point $P_1,P_2,P_3$ if and only if $\alpha(c_1,c_2)$ is
not   present  in  our  linear  combination,  i.e.  if  and  only  if
$c_1=c_2=0$. This means that
$$
\phi(\omega)=\phi(\gamma)=
\dim\operatorname{Vect}(\beta,\gamma)=2\equiv 0\pmod 2
$$
and our Lemma is proved.
\end{proof}

\section{An exercise in arithmetics}
\label{a:exercise}

In  this  appendix  we  present  an exercise promised in the proof of
Theorem \ref{thm:1prime}.  We show that if a square-tiled cyclic cover $M_N(a_1,a_2,a_3,a_4)$ has
a  symmetry represented by a single cycle of length $4$ or $2$, then
the index of the Veech group is different from $1$.

We  start  with  a symmetry represented by a cycle of length $4$. If
the  index  of  the  Veech  group  is  $1$, then for all permutations
$\pi\in\mathfrak{S}_4$  the flat surfaces represented by square-tiled
cyclic  covers $M_N(a_{\pi(1)},a_{\pi(2)},a_{\pi(3)},a_{\pi(4)})$ are
isomorphic.  By  Lemma~\ref{lm:cylinder:width}  this implies that for
any  $i\neq  j$  and any $m\neq l$, where $i,j,m,l\in\{1,2,3,4\}$ one
has
\begin{equation}
\label{eq:gcdij}
\gcd(N,a_i+a_j)=\gcd(N,a_m+a_l)\,.
\end{equation}
Since  the  symmetry  group is the entire $\mathfrak{S}_4$, we may assume
that  the ramification points are numbered in such way that the cycle
$\pi$ acts as
$$
a_1\to a_2\to a_3\to a_4\to a_1\,,
$$
that is
$$
a_i=a_1\cdot     k^{i-1}\pmod{N}, \text{ where }  i=1,2,3,4\,.
$$
Conditions \eqref{eq:a1:a4}   imply   that   $\gcd(a_1,N)=1$.  Hence,
\begin{align}
\notag
\gcd(a_1+a_2,N)&=\gcd(a_1\cdot(1+k),N)=\gcd(1+k,N)\\
\label{eq:three:gcd}
\gcd(a_1+a_3,N)&=\gcd(a_1\cdot(1+k^2),N)=\gcd(1+k^2,N)\\
\notag
\gcd(a_2+a_3,N)&=\gcd(a_1\cdot(k+k^2),N)=\gcd(k+k^2,N)\,.
\end{align}
By \eqref{eq:gcdij} we have
$$
\gcd(a_1+a_2,N)=\gcd(a_1+a_3,N)=\gcd(a_2+a_3,N):=d\,.
$$
Clearly
$$
\gcd((a_2+a_3)-(a_1+a_3)),N)=
\gcd(a_1\cdot(k-1),N)=
\gcd(k-1,N)
$$
is   divisible   by  $d$.  Since  $\gcd(k+1,N)=d$  we  conclude  that
$2=(k+1)-(k-1)$ is divisible by $d$, and, hence, $d$ is either $1$ or
$2$.

By \eqref{eq:a1:a4}  $a_1+a_2+a_3+a_4$  is  divisible  by $N$. On the
other hand,
\begin{multline*}
\gcd(a_1+a_2+a_3+a_4,N)=
\gcd\big((a_1(1+k)(1+k^2),N\big)
=\\=
\gcd\big((1+k)(1+k^2),N\big)
\end{multline*}
divides the product $\gcd(1+k,N)\cdot\gcd(1+k^2,N)$ which is equal to
$d^2$  by \eqref{eq:three:gcd}.  Hence, $N$  is  one  of the integers
$1,2,4$,  which  contradicts  the  assumption that all $a_i, a_j$ are
pairwise distinct.

Let us suppose now that  $M_N(a_1,\dots,a_4)$  has  a symmetry represented by a cycle of
length $2$  and  let  us show that the Veech group cannot have index
$1$.  Since  the  symmetry  group is the entire $\mathfrak{S}_4$, without
loss of generality we may assume that
\begin{align*}
k\cdot a_1\pmod{N}&=a_1\\
k\cdot a_2\pmod{N}&=a_2\\
k\cdot a_3\pmod{N}&=a_4\\
k\cdot a_4\pmod{N}&=a_3
\end{align*}

Let   $\ell:=N/\gcd(k-1,N)$.  Since  $(k-1)\cdot  a_1\pmod{N}=0$  and
$a_1\le  N$,  $1<k<N$, we  have  $\ell>1$,  and $\ell$ divides $a_1$.
Similarly,  $\ell$  divides $a_2$. Hence $\ell$ divides $a_1+a_2$ and
thus  $\gcd(a_1+a_2,N)$.  We  have  seen  that, when the index of the
Veech   group   is $1$,  the relations in formula \eqref{eq:gcdij}  hold,
in particular
$$
\gcd(a_1+a_2,N)=\gcd(a_1+a_3,N)\,.
$$
Since  $\ell$  is  a  divisor of both $\gcd(a_1+a_3,N)$ and $a_1$, it
divides  $a_3$, hence it also divides $a_4=k\cdot a_3\pmod{N}$.
Thus,  $\ell$ divides $\gcd(N,a_1,a_2,a_3,a_4)$. Since $\ell>1$, this
contradicts the second condition in formula \eqref{eq:a1:a4}.

\acknowledgements

We  would  like  to thank J.-C. Yoccoz, who has motivated our work by
asking   whether   the   example   in \cite{ForniSurvey}   could   be
generalized.  We  are  also  very grateful to M. M\"oller who told us
about  cyclic  covers,  directed  us to the reference \cite{Bouw} and
suggested  the  correct  algebraic equation of one of our examples.
A. Zorich also thanks A. Eskin and \mbox{M. Kontsevich} for the pleasure
to  work  with  them on the project \cite{Eskin:Kontsevich:Zorich} on
\mbox{Lyapunov}  exponents;  in  particular,  on  its part related to
evaluation  of  individual  \mbox{Lyapunov} exponents of square-tiled
cyclic covers \cite{Eskin:Kontsevich:Zorich:cyclic}.

We thank an anonymous referee, V. Delecroix, P. Hubert,
and A. Wright for careful reading the manuscript  and  for  indicating  to
us some typos, and a gap in the initial version of Theorem \ref{th:SLZ:orbit}.

The authors thank Coll\`ege de France, IHES, HIM and MPIM for their
hospitality while preparation of this paper.

\end{document}